\documentclass[a4paper,11pt]{article}
\usepackage[utf8]{inputenc}

\usepackage{amsthm}
\usepackage{amsmath}
\usepackage{amssymb}
\usepackage{bbm}
\usepackage{mathtools}
\usepackage[pdfencoding=auto,hyperindex,pdftex]{hyperref}
\usepackage{fancyhdr}
\usepackage{pdfpages}
\usepackage{MnSymbol}
\usepackage{subcaption}
\usepackage{tikz}
\usetikzlibrary{math, calc,fpu}

\newcommand{\supp}{\mathrm{supp\,}}

\newcommand{\INT}{\mathrm{int}}
\newcommand{\x}{\times}
\newcommand{\p}{\partial}

\newcommand{\id}{\mathrm{Id}}
\renewcommand{\d}{\mathrm{d}}

\newcommand{\eqnb}{\begin{equation}}
\newcommand{\eqnbs}{\begin{equation*}}
\newcommand{\eqnbsa}{\begin{equation*}\begin{aligned}}
\newcommand{\eqnba}{\begin{equation}\begin{aligned}}
\newcommand{\eqnbl}[1]{\begin{equation}\label{#1}}
\newcommand{\eqnbal}[1]{\begin{equation}\label{#1}\begin{aligned}}
\newcommand{\eqnes}{\end{equation*}}
\newcommand{\eqne}{\end{equation}}
\newcommand{\eqnesa}{\end{aligned}\end{equation*}}
\newcommand{\eqnea}{\end{aligned}\end{equation}}
\newcommand{\loc}{\mathrm{loc}}
\newcommand{\re}[1]{(\ref{#1})}
\newcommand{\comment}[1]{}

\newcommand{\partref}[2]{\hyperref[#1]{\ref*{#1}#2}}

\newcommand{\FiguresPath}{}

\newcommand{\RR}{\mathbb{R}}

\newcommand{\cR}{\mathcal{R}}

\newcommand{\NN}{\mathbb{N}}

\newcounter{cntNotation}

\newtheorem{theorem}{Theorem}
\newtheorem{proposition}[theorem]{Proposition}
\newtheorem{definition}[theorem]{Definition}
\newtheorem{notation}[cntNotation]{Notation}
\newtheorem{lemma}[theorem]{Lemma}
\newtheorem{corollary}[theorem]{Corollary}

{\bf}{\it}

\title{Non-conservation of dimension in divergence-free solutions of passive and active scalar systems}
\author{Charles L. Fefferman$^1$,  Benjamin C. Pooley$^2$ \& Jos\'e L. Rodrigo$^3$}
\date{\it$^1$Department of Mathematics, Princeton University,\\ Princeton NJ, 08544\\
$^{2,3}$Mathematics Institute, University of Warwick,\\ Coventry, CV4 7AL}
\begin{document}
\maketitle

\abstract{For any $h\in(1,2]$, we give an explicit construction of a compactly supported, uniformly continuous, and (weakly) divergence-free velocity field in $\RR^2$ that weakly advects a measure whose support is initially the origin but for positive times has Hausdorff dimension $h$. 

These velocities are uniformly continuous in space-time and compactly supported, locally Lipschitz except at one point and satisfy the conditions for the existence and uniqueness of a  Regular Lagrangian Flow in the sense of Di Perna and Lions theory.

We then construct active scalar systems in $\RR^2$ and $\RR^3$ with measure-valued solutions whose initial support has co-dimension 2 but such that at positive times it only has co-dimension 1. The associated velocities are divergence free, compactly supported, continuous, and sufficiently regular to admit unique Regular Lagrangian Flows.

This is in part motivated by the investigation of dimension conservation for the support of measure-valued solutions to active scalar systems. This question occurs in the study of vortex filaments in the three-dimensional Euler equations.  
}
\section{Introduction}

The classical transport equation describing the advection of a quantity $\omega:[0,T]\x\RR^d\to\RR$ by a velocity $u:[0,T]\x\RR^d\to\RR^d$ is given by
\begin{equation}\label{eqTransportClassical}
\p_t\omega+(u\cdot\nabla)\omega=0.
\end{equation}
In many examples this may form part of either a {\it passive scalar} or an {\it active scalar} system depending on whether or not $u$ is dependent on $\omega$. In the case that $u$ is divergence free (and sufficiently regular), this system is equivalent to a continuity equation:
\[
\p_t\omega+\nabla\cdot (\omega u)=0.
\]

The connection between the Eulerian and Lagrangian viewpoints for \re{eqTransportClassical} have been widely studied. In the classical theory, if $u$ is uniformly Lipschitz, or satisfies an Osgood-type condition then the Lagrangian trajectories defined by
\begin{equation}\label{eqLagClassical}
\frac{\d }{\d t} X(t,s,a)=u(t,X(t,s,a)),\quad X(s,s,a)=a
\end{equation}
exist for all $a\in\RR^d$ and are unique (see  \cite{Hartman_Book}, for example). In this case, $X(t,s,a)$ are the characteristics for \re{eqTransportClassical}, i.e\  $\omega(t,X(t,s,a))=\omega(s,a)$. To save notation, we will usually write $X(t,0,a)\eqqcolon X(t,a)$.

Di Perna and Lions \cite{DiPerna_Lions_1989}, proved that if $u\in L^1(0,T;W^{1,1}_\loc(\RR^d))$ with $\nabla\cdot u\in L^1(0,T;L^\infty)$ and
\begin{equation}\label{eqDLgrowth_condition}
\frac{|u|}{1+|x|}\in L^1(0,T;L^1)+L^1(0,T;L^\infty)
\end{equation}
then there exists a semigroup  $Y(t+\tau,s,\cdot)=Y(t+\tau,\tau,Y(\tau,s,\cdot))$, with $Y(s,s,a)=a$, that define a trajectory map in the sense that \[
(t,Y(t,a))\mapsto\omega_0(a)
\] is a distributional solution of \re{eqTransportClassical}, for any  $\omega_0$ in a particular subspace of $C^1(\RR^2)$. This solution is unique subject to growth and decay bounds on the push-forward of the Lebesgue measure $X\#\lambda$. For these solutions we also have that \re{eqLagClassical} holds for almost every $a$ where the time derivative is taken in the sense that $t\mapsto X(t,a)$ is absolutely continuous, i.e.
\begin{equation}\label{eqLagAC}
X(t,a)=a+\int_0^t u(s,X(s,a))\,\d s
\end{equation}
for all $t\in[0,T]$ and almost all $a$.

These results have been extended to the case of less regular vectorfields, for example by Ambrosio and co-authors, see \cite{Ambrosio_2017} and references therein. In particular, the existence and uniqueness theory for so-called regular Lagrangian flows extends to $u\in L^1(0,T; BV_\loc)$ satisfying \re{eqDLgrowth_condition} and $(\nabla\cdot u)^-\in L^1(0,T;L^\infty)$ \cite{Ambrosio_2004a}.

An example of non-uniqueness of trajectory maps for a non-$BV$ fields in $\RR^d$, $d\geq 3$ was already known due to Aizenman \cite{Aizenman_1978b}. In brief, he combined rescaled copies of a velocity whose associated flow was piecewise affine at $t=1/4$ to construct a measure-preserving velocity that mapped line segments of the form $[0,1]\x\{y\}\x\{1\}$ to points in $[0,1]\x[0,1]\x\{0\}$ in finite time, thus allowing any measurable permutation of trajectories on a line segment at time $t=1/2$. 

Also of note are results giving conditions on a velocity to guarantee that almost evey trajectory avoids a set, depending on its (co)dimension. Aizenman showed that a measure-preserving flow $X$, corresponding to a time-independant $u$ in the sense of \re{eqLagAC} {\it avoids} a set $A$, in a specified sense, if $u\in L^p(\RR^d)$ and
\[
\frac{1}{p}+\frac{1}{C(A)}<1
\]
where $C(A)$ is essentially the co-box-dimension $d-\dim_B(A)$.

This has been extended to time-dependent flows by Robinson\ et.\ al. \cite{JCR_Sharples_2013} using the notion of {\it $r$-dimensional prints}.  \cite{JCR_WS_2009},\cite{JCR_Sad_2009b}, and \cite{JCR_Sad_Sharples_2013}, in which such  sufficient conditions for a flow to avoid a subset of $[0,\infty)\x\RR^3$ are combined with partial regularity results for the Navier--Stokes equations to yield uniqueness of almost every trajectory for {\it suitable weak solutions}.

In Section \ref{secPassiveScalarH} we construct a uniformly continuous, compactly supported, divergence-free, and time-dependent velocity in $\RR^2$, such that trajectory of the origin is not unique in the following sense. There is a time-dependent family of measures $\omega(t)$ with $\omega(0)=\delta_0$ that is advected by $u$ in a distributional sense, and $\dim_H\supp\omega(t)=2$ for $t>0$. We show that we may even take $u$ locally Lipschitz in $[0,1]\x\RR^2$ away from $(0,(0,0))$, in the case that $\dim_H\supp\omega(t)=h\in[1,2)$.

The velocities constructed in this way naturally fall within the regime of Di Perna and Lions, as they are weakly divergence-free, compactly supported and belong to $L^q(0,1;W^{1,p})$ for a specific range of vaules $p\in[1,\infty)$, $1\leq q\leq\infty$ (see Propositions \ref{propWregularity} and \ref{propVhatRegularity}). In particular we may take $p=q=1$ in all examples. Therefore, the constructed velocities each admit a unique regular Lagrangian flow in the sense of Di Perna and Lions, despite the fact that the dimension of advected sets may have jump discontinuities.

In Sections \ref{secActiveScalars2D} and \ref{secActiveScalars3D}, we start by considering active scalar systems in two dimensions i.e.\ \re{eqTransportClassical} coupled with the relation
\begin{equation}\label{eqUConvolution}
u(t,x)=\int_{\RR^2} K(x-y)\ \d\omega(t)(y),
\end{equation}
for some continuous kernel $K$.  We choose $K$ so that there exists a solution $\omega$ with $\omega(0)=\delta_0$ and $\dim_H\supp\omega(t)=1$ for $t>0$. We do this in such a way that $u$ is uniformly continuous, compactly supported and divergence free. The resulting velocities $u$ also belong to $L^\infty(0,T;W^{1,1}(\RR^d))$ by Proposition \ref{propSlitURegularity}. Hence the velocity also admits a regular Lagrangian flow, in the sense of Di Perna and Lions.

Finally, we adapt the two-dimensional example to construct a vector-valued $\omega$ in three dimensions satisfying a transport equation with stretching:
\begin{equation}\label{eqTransWStretching}
\p_t u+(u\cdot\nabla)\omega=(\omega\cdot\nabla)u,\quad u=K\ast\omega
\end{equation}
for a matrix-valued $K$. In this example $\dim_H\omega(0)=1$ but $\dim_H\omega(t)=2$ for $t>0$. 

We are partly motivated by the problem of analysing the evolution of isolated vortex filaments for the three-dimensional Euler equations, which can be written in the form \re{eqTransWStretching} using the Biot--Savart kernel (see \cite{MajdaBertozzi},\cite{JCR_NSE_Book}):
\begin{equation}\label{eqBS3D}
u=-\frac{1}{4\pi}\int_{\RR^3}\frac{(x-y)\x\omega(y)}{|x-y|^3}\ \d y.
\end{equation}

Formal asymptotics suggest that if vorticity has a $\delta$-distribution as the tangent to a curve, then to leading order, each point on the curve evolves in the direction of its binormal (in the Frenet-Serret sense) at a rate proportional to the curvature, in a rescaled asymptotic sense. See \cite{ArmsHama_1965} for an example of such classical arguments.

For recent progress on the vortex filament problem for the three-dimensional Euler equations and the binormal curvature flow itself, see for example \cite{JerrardSeis_2017}, \cite{BanicaVega2015}, \cite{HozVega2018}.

In \cite{PooleyRodrigo2019a} Pooley and Rodrigo derive the asymptotics for a family of models of the Euler equations in which vortex filaments have finite velocity along the binormal to leading order. This leads to the natural problem of determining sufficient conditions for a velocity field to flow filaments to filaments. In Section \ref{secActiveScalars3D} we show that it is not sufficient for the velocity to be divergence free, continuous and $W^{1,p}(\RR^3)$ even when the velocity is generated by the filament in the sense of \re{eqTransWStretching}. 

\medskip
To finish this section we now set up and state our main results. Our results in two dimensions will make use of the following definition of weak measure-valued solutions of the transport equation \re{eqTransportClassical}.

\begin{definition}\label{defWeakAdvectionMeasure2D}
A time-dependent locally finite Borel
 measure $\mu=\mu(t)$ is \emph{weakly advected} by a continuous weakly divergence-free velocity $u\in C([0,T]\x\RR^2;\RR^2)$ if
\begin{multline}\label{eqWeakAdvectionMeasure2D}
\int_0^T\int\p_t\phi(t,x)\,\d\mu(t)(x)\d t- \int \phi(T,x)\,\d\mu(T)(x) \\
+  \int \phi(0,x)\,\d\mu(0)(x)+\int_0^T\int(u\cdot\nabla)\phi\,\d\mu(t)(x)\d t =0,
\end{multline}
for all $\phi\in C^\infty_c([0,T]\x\RR^2)$.
\end{definition}

In Section \ref{secPassiveScalarH} we prove the first main result:
\begin{theorem}\label{thmPassiveScalarH}
For any $h\in(1,2)$ there exists a uniformly continuous divergence-free velocity $u\in C_c([0,\infty)\x\RR^2)$ and a time-dependent measure $\omega$ that is weakly advected by $u$, such that $\supp(\omega(0))=\{(0,0)\}$ but the Hausdorff dimension $\dim_H\supp(\omega(t))=h$ for all $t>0$.

Additionally, $u$ is locally Lipschitz in $\left([0,\infty)\x\RR^2\right)\backslash\{(0,(0,0))\}$.
\end{theorem}
In Subsection \ref{secPassiveScalar2} we also see that if instead $u$ is only locally Lipschitz on $[0,\infty)\x(\RR^2\backslash\{0\})$ we can extend this to the case $h=2$.

\begin{theorem}\label{thmPassiveScalar2}
There exists a uniformly continuous divergence-free velocity $u\in C_c([0,\infty)\x\RR^2)$ and a measure $\omega$ that is advected by $u$ weakly, such that $\supp(\omega(0))=\{(0,0)\}$ but $\dim_H\supp(\omega(t))=2$ for all $t>0$.

Additionally $u$ is locally Lipschitz in $[0,\infty)\x(\RR^2\backslash\{(0,0)\})$.
\end{theorem}

In Section \ref{secActiveScalars2D} we begin to consider active scalar systems. That is, we add the requirement that the velocity $u$ be recovered from the measure $\omega$ by a convolution:

\begin{theorem}\label{thmActiveScalar2D}
There exists a kernel $K\in C(\RR^2;\RR^2)$ and $T>0$ such that there exists a time-dependent measure $\omega$ that is weakly advected by 
\begin{equation}\label{eqUdefn_abstractOmega2D}
u(t,x)\coloneqq \int_{\RR^2} K(x-y)\,\d\omega(t)(y),
\end{equation}
and $\dim_H\supp\omega(0)=0$, $\dim_H\supp\omega(T)=1$.
\end{theorem}

Finally,  in Section \ref{secActiveScalars3D}, we extend the construction in Section \ref{secActiveScalars2D} to the three-dimensional vector-valued case with stretching. That is, we consider weak solutions $\omega$ of the system
 \[
 \p_t u+(u\cdot\nabla)\omega=(\omega\cdot\nabla)u,
 \]  
 for given $u$. In this case we use the following definition of weak measure-valued solutions. 
 \begin{definition}\label{defWeakAdvectionMeasure}
A time-dependent vector-valued locally finite Borel measure $\omega$ with locally finite distributional divergence is weakly advected by a continuous velocity $u:[0,T]\x\RR^3\to\RR^3$ if
\begin{multline}\label{eqWeakAdvectionMeasure}
\int_0^T\int\p_t\phi(t,x)\cdot\,\d\omega(t)(x)\d t- \int \phi(T,x)\cdot\,\d\omega(T)(x) \\
+  \int \phi(0,x)\cdot\,\d\omega(0)(x)+\int_0^T\int((u\cdot\nabla)\phi)\big|_{(t,x)}\cdot\,\d\omega(t)(x)\d t\\
-\int_0^T\int ((\nabla\phi)^\top u)\big|_{(t,x)}\cdot\,\d\omega(t)(x)\d t - \int_0^T\langle\phi\cdot u,\nabla\cdot\omega(t)\rangle\d t =0,
\end{multline}
for all $\phi\in C^\infty_c([0,T]\x\RR^3;\RR^3)$.
\end{definition}

In this context we prove the following.

\begin{theorem}\label{thmTimeReversal}
There exists a kernel $K\in C(\RR^3;\RR^{3\x3})$, $T>0$, and a time-dependent and locally finite measure $\omega$ with locally finite distributional divergence that is weakly advected by 
\begin{equation}\label{eqUdefn_abstractOmega}
u(t,x)\coloneqq \int K(x-y)\,\d\omega(t)(y),
\end{equation}
such that $\dim_H\supp\omega(0)=1$, $\dim_H\supp\omega(T)=2$, and $u$ is weakly divergence free.
\end{theorem}

\section{Flows with fractal structure}\label{secPassiveScalarH}
In this section we prove Theorem \ref{thmPassiveScalarH} and Theorem \ref{thmPassiveScalar2}. In each case we begin by constructing a velocity that advects a set of specified dimension into the origin in finite time. From the corresponding trajectory maps (defined below) restricted to the set of interest, a measure can be constructed that is weakly advected by $u$, in the sense of Definition \ref{defWeakAdvectionMeasure2D}. After a time-reversal argument, these weak measure-valued solutions yield examples where $\dim_H\supp \omega(t)$ has the required jump at the initial time.

\begin{definition}\label{defTrajectoryMap}
For $d\in\NN$, given $U\in C([0,T]\x\RR^d;\RR^d)$ and $A\subset\RR^d$, we say that $X_U:[0,T]\x A\to\RR^d$ is a \emph{trajectory map} for  $U$ (on $A$) if $t\mapsto X_U(t,a)$ is differentiable for any $a\in A$ at all $t\in[0,T]$ and
\begin{equation}\label{eqDefnOfLagFlow}
\p_tX_U(t,a)=U(t,X_U(t,a)),\quad X_U(0,a)=a.
\end{equation}
It should be assumed that  $A=\RR^d$ unless specified.
\end{definition}

For any fixed $\alpha \in (1/2,1/\sqrt{2})$ we construct a particular inhomogeneous self-similar set $S_\alpha$ with $\dim_HS_\alpha=-\log 2/\log\alpha\in(1,2)$. The construction is such that we can exhibit an explicit example of a divergence free, time-dependent velocity $W$ with $X_W(t,S_\alpha)=\{0\}$ for all sufficiently large $t$.

The following lemma allows us to pass to the weak formulation from constructions based on trajectory maps. The proof is not difficult and is omitted. A similar, but more involved proof is presented in Lemma \ref{lemPassToWeak}.
\begin{lemma}\label{lemPassToWeak2D}
Let $\mu_0$ be a finite Borel measure with 
\[\supp\mu_0=A\subset\RR^2\] and let $u\in C([0,T]\x\RR^2;\RR^2)$ admit a trajectory map $X_u$ on $A$, for some $T>0$. Then the measure $\mu=\mu(t)$, defined by the push-forward of $\mu_0$ under $X_u$,
\[
\mu(t)\coloneqq X_u(t)\#\mu_0
\]
 is weakly advected by $u$ on the time interval $[0,T]$.
\end{lemma}

In the weak formulation the solutions of the transport equation are time reversible, as detailed in the following lemma. The proof follows directly from Definition \ref{defWeakAdvectionMeasure2D} and is omitted.
\begin{lemma}\label{lemReversibility}
If a family of locally finite Borel measures $\mu(t)$ is weakly advected by $u\in C([0,T]\x\RR^2;\RR^2)$ then $\tilde\mu(t)\coloneqq\mu(T-t)$ is weakly advected by $\tilde u(t)\coloneqq -u(T-t)$ on $[0,T]$.
\end{lemma}

Applying Lemmas \ref{lemPassToWeak2D} and \ref{lemReversibility} to the $h$-dimensional Hausdorff measure restricted to $S_\alpha$:
\[
\mathcal{H}^h_{S_\alpha}\coloneqq \mathcal{H}^h\lefthalfcup S_\alpha,
\] one can check that Theorem \ref{thmPassiveScalarH} is a consequence of the following:

\begin{theorem}\label{thmFractal1}
For $h\in(1,2)$ there exists a uniformly continuous, weakly divergence-free, velocity $W:[0,1]\x\RR^2\to\RR^2$ and a compact set $S\subset\RR^2$ with Hausdorff dimension $\dim_H(S)=h$ such that $W$ admits a unique  Lagrangian flow $X_W$, which satisfies:
\[
\dim_HX_W(t,S)=h\quad \forall t\in[0,1),\qquad X_W(1,S)=\{0\}.
\]
\end{theorem}

\subsection{Notation and a family of maps}
We begin the proof of Theorem \ref{thmFractal1} by setting out some notation. For fixed $h\in(1,2)$ let $\alpha=\alpha_h\coloneqq 2^{-1/h}$, denote a scale factor used in constructing the self-similar set $S_\alpha$. These choices are illustrated in Figure \ref{figRFR}.

\begin{notation}\label{notConstants}
Given $\alpha\in (1/2,1/\sqrt{2})$, fix the following constants for the construction of the flow:
\[
\varepsilon=\varepsilon_\alpha \coloneqq \frac{1}{4\alpha}-\frac{1}{2\sqrt{2}}<\frac{1}{4},
\]
\[
\delta=\delta_\alpha\coloneqq \frac{1}{4}-\frac{\alpha}{2\sqrt{2}}=\alpha\varepsilon_\alpha,
\]
are mollification/cutoff radii, and 
\[
\gamma=\gamma_\alpha\coloneqq 1-\delta_\alpha= \frac{3}{4}+\frac{\alpha}{2\sqrt{2}}
\]
is a contraction ratio, associated to the flow.
\end{notation}
\begin{notation}
The set of finite binary words is denoted by \[
\{1,2\}^\ast\coloneqq\bigcup_{k\in\NN_0}\{1,2\}^k,
\]
where $\NN_0=\NN\cup\{0\}$ and by convention we take $\{1,2\}^0=\{\emptyset\}$.
\end{notation}
\begin{notation}
Let $\eta\in C^\infty_c(\RR)$ be non-negative and compactly supported on $(0,1)$, such that $\int_0^1\eta =\delta$, for $\delta=\delta_\alpha$ as above.
We define the following families of affine bijections on $\RR^2$:
\[
F_1^t(x)\coloneqq \left(1-\int_0^t\eta,0\right)+\alpha \mathcal{R}(x),\quad F_2^t(x)\coloneqq\left(-1+\int_0^t\eta,0\right)+\alpha \mathcal{R}(x)
\]
 for $t\in\RR$, where $\mathcal{R}$ denotes a (counter clockwise) rotation by $\pi/2$ about the origin, i.e.\ $\mathcal{R}x=x^\perp$.
 Furthermore, for $k\in\NN$ and $w\in \{1,2\}^k$ define
 \[F^t_w\coloneqq F_{w_1}^t\circ\ldots\circ F_{w_k}^t.\]
 We also take $F^t_\emptyset=\id$.
\end{notation}
We will always assume that $\|\eta\|_{L^\infty}\lesssim\delta$. That is, $\|\eta\|_{L^\infty}\leq C\delta$ for some universal constant $C>0$.

\begin{notation}
For $\xi\geq0$ we denote by $R_\xi$ the closed rectangle
\[
R_\xi\coloneqq[-2-\xi,2+\xi]\x[-\sqrt{2}-\xi,\sqrt{2}+\xi].
\]
\end{notation}
\subsection{The self-similar set \texorpdfstring{$S_\alpha$}{Sɑ}}
Denote by $I$ the line segment $I=[-1,1]\x\{0\}$. For $\alpha\in(1/2,1/\sqrt{2})$ consider the affine linear maps
\[
F_1(x)=F_1^0(x)=(1,0)+\alpha \mathcal{R}(x),\quad F_2(x)=F_2^0(x)=(-1,0)+\alpha \mathcal{R}(x),
\] (see Figure \ref{figRFR}). 
\begin{figure}[b]
\centering
\includegraphics[width=10cm] {\FiguresPath 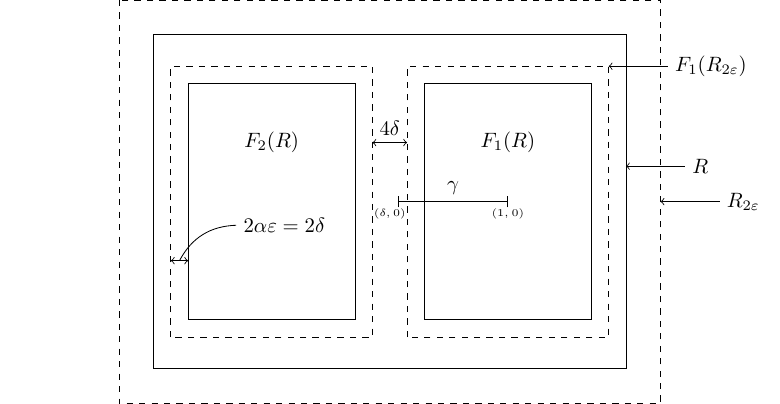}
\caption{Geometric significance of $\delta$, $\varepsilon$ and $\gamma$.}\label{figRFR}
\end{figure}
Now let $S_\alpha$ be the compact attractor of an inhomogeneous iterated function system: 
\begin{equation}\label{eqSAlphaDefn}
S_\alpha=I\cup \bigcup_{i=1,2}F_i(S_\alpha).
\end{equation}
It is straightforward to check that $S_\alpha$ is given by
\begin{equation}\label{eqFracCharacter}
S_\alpha =\overline{\bigcup_{w\in\{1,2\}^\ast}F_w(I)}.
\end{equation}
See \cite{BFM_2019} and references therein for further discussion of inhomogeneous iterated function systems and the calculation of their attractors.

Denote by $\sigma_1$ and $\sigma_2$ the reflections in the $x$ and $y$ axes respectively. The proof of the following lemma is elementary.
\begin{lemma}The attractor $S_\alpha$ is invariant with respect to the reflections $\sigma_1$, and $\sigma_2$. 
\end{lemma}
\begin{figure}
\centering
\begin{subfigure}{0.45\textwidth}
\centering
\includegraphics{\FiguresPath 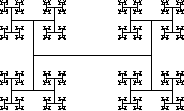}
\caption{$\alpha=0.6$}
\end{subfigure}
\begin{subfigure}{0.45\textwidth}
\centering
\includegraphics{\FiguresPath 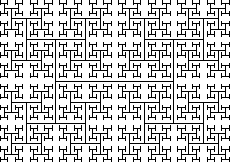}
\caption{$\alpha=0.707$}
\end{subfigure}
\caption{$S_\alpha$}
\end{figure}
\subsubsection{The dimension of \texorpdfstring{$S_\alpha$}{Sɑ}}
We next calculate the Hausdorff dimension of $S_\alpha$.
\begin{notation}
Let $H_\pm$ denote the closed half-spaces
\[
H_\pm\coloneqq\{x\,\colon \pm x_1\geq 0\}.
\]
For $\xi\in\RR$ let
\[
H_\pm^{\xi}\coloneqq\{x\,\colon \pm x_1\geq \xi\}.
\]
\end{notation}

\begin{lemma}$S_\alpha\subset \left(I \cup F_1(R_0)\cup F_2(R_0)\right)\subset R_0$.
\end{lemma}
\begin{proof}
By \re{eqFracCharacter} and since $I\subset R_0$, it suffices to check that $F_i(R_0)\subset R_0$ for $i=1,2$. Indeed, from this it will follow that $F_w(R_0)\subset R_0$ for all non-empty finite words $w\in \{1,2\}^\ast$.

Now \begin{multline*}
F_1(R_0)=[1-\sqrt{2}\alpha ,1+\sqrt{2}\alpha]\x [-2\alpha,2\alpha]\\
=[4\delta ,1+\sqrt{2}\alpha]\x [-2\alpha,2\alpha]\subset R_0
\end{multline*}
since $\alpha<1/\sqrt{2}$. The fact that $F_2(R_0)\subset R_0$ follows by a similar argument, or alternatively by symmetry.
\end{proof}
The above calculation shows the following, in fact.
\begin{corollary}\label{corHalfSpaceS}
\begin{equation}\label{corHalfSpaceS:eq1}F_1(S_\alpha)\subset F_1(R_0)\subset H_+^{4\delta}
\end{equation} and
\begin{equation}\label{corHalfSpaceS:eq2}F_2(S_\alpha)\subset F_2(R_0)\subset H_-^{4\delta}.
\end{equation}
\end{corollary}

As $S_\alpha$ is an inhomogeneous self-similar set, we have  (see \cite{BFM_2019}) that 
\[\dim_HS_\alpha = \max\left(\dim_H \bigcup_{w\in \{1,2\}^\ast}F_w(I),\dim_H \widetilde{S}_\alpha\right) = \max(1,\dim_H \widetilde{S}_\alpha),\] where we use the countable stability of the Hausdorff dimension to estimate the dimension of the first term. Here $\widetilde{S}_\alpha$ is the attractor of the homogeneous  system:
\[
\widetilde{S}_\alpha=F_1(\widetilde{S}_\alpha)\cup F_2(\widetilde{S}_\alpha).
\]
By Corollary \ref{corHalfSpaceS}, $\widetilde{S}_\alpha$ satisfies a strong separation condition, so the Hausdorff dimension agrees with the similarity dimension. Thus
\begin{equation}\label{eqDim}
\dim_H\widetilde{S}_\alpha = -\log(2)/\log(\alpha).
\end{equation}
It follows that $\dim_H S_\alpha=\dim_H\widetilde{S}_\alpha=h\in (1,2)$ since \[\alpha=2^{-1/h}\in(1/2,1/\sqrt{2}).\]

\subsubsection{Properties of \texorpdfstring{$F_w(R_\varepsilon)$}{Fw(Rε)}}
We now consider the images of $R_\varepsilon$ under $F_w$ for $w\in \{1,2\}^\ast$, the simple facts here will clarify the velocity constructions in the next section.

Calculations similar to those used to prove Corollary \ref{corHalfSpaceS} yield:
\begin{lemma}\label{factHalfSpace2} For $t\in\RR$,
\begin{equation}\label{factHalfSpace2:eq1}F_1^t(R_0)\subset F_1^t(R_\varepsilon)\subset \INT\left(R_0\cap H_+^{\delta}\right)
\end{equation} and
\begin{equation}\label{factHalfSpace2:eq2}F_2^t(R_0)\subset F_2^t(R_\varepsilon)\subset \INT\left(R_0\cap H_-^{\delta}\right).
\end{equation}
\end{lemma}

As a consequence we have the following result about the disjointness of the images of $\overline{R_\varepsilon\backslash \left(R_0\cap \left[H_+^\delta\cup H_-^\delta\right]\right)}$, which will contain the supports of the the gradients of the velocities we construct later.
\begin{lemma}\label{factDisjointBorders}
For any $w,w'\in \{1,2\}^\ast$ with $w\neq w'$ and for any $t\in \RR$
\[
F_w^t\left(\overline{R_\varepsilon\backslash \left(R_0\cap \left[H_+^{\delta}\cup H_-^{\delta}\right]\right)}\right)\cap F_{w'}^t\left(\overline{R_\varepsilon\backslash \left(R_0\cap \left[H_+^{\delta}\cup H_-^{\delta}\right]\right)}\right)=\emptyset
\]
\end{lemma}
\begin{proof}
Without loss of generality, suppose that $w\in\{1,2\}^k$ and $w'\in\{1,2\}^{k'}$ with $k'\geq k\geq 0$ and $k'>0$. There are two cases to consider: either there exists $0<j\leq k$ such that $w_j\neq w'_j$, or $k<k'$ and $(w'_1,\ldots,w'_k)=w$.

In the first case, we may assume that $w_i=w_i'$ for $i<j$. By \re{factHalfSpace2:eq1} and \re{factHalfSpace2:eq2},
\[
F_{(w_1,\ldots,w_{j-1})}^t\circ F_1^t(R_\varepsilon)\subset F_{(w_1,\ldots,w_{j-1})}^t(R_0\cap H^\delta_+),
\]
and
\[
F_{(w_1,\ldots,w_{j-1})}^t\circ F_2^t(R_\varepsilon)\subset F_{(w_1,\ldots,w_{j-1})}^t(R_0\cap H^\delta_-),
\]
which are disjoint since $F_1^t$ and $F_2^t$ (and $F^t_\emptyset=\id$, if $j=1$) are injective.

In the second case, suppose that $w'_{k+1}=1$, then
\[
F_{w'}^t(R_\varepsilon)=F_{w}^t\circ F_{1}^t\circ F_{(w_{k+2}',\ldots, w_{k'}')}^t(R_\varepsilon)\subset F_{w}^t(\INT(R_0\cap H_+^\delta)).
\]
By injectivity of $F_w^t$, the right-hand side is disjoint from \[F_w^t\left(\overline{R_\varepsilon\backslash \left(R_0\cap \left[H_+^{\delta}\cup H_-^{\delta}\right]\right)}\right),\] as required. The case $w'_{k+1}=2$ is similar.
\end{proof}

\subsection{The velocity field}
In this section, we define the divergence-free velocity that will collapse $S_\alpha$ to the origin in finite time. It will be constructed from a fundamental (time-independent) local flow. This latter vectorfield is given by applying $\nabla^\perp$ to a mollified and smoothly cutoff piecewise linear function, thus it is automatically divergence free. 

Define the piecewise linear function 
\[
\Psi(x)=\left\{\begin{array}{ll}
x_2&x_1>0\\
-x_2&x_1<0\\
0&x_1=0
\end{array}
\right. 
\]
and let
\[
u=u_\alpha\coloneqq\nabla^\perp (\chi J_{\delta}\Psi),
\]
where $\chi$ is a smooth cutoff supported on $R_{\varepsilon}$ and identically $1$ on $R_0$. The operator $J_\delta$ denotes convolution with a compactly supported mollifier $\rho_\delta(x)=\frac{1}{\delta^2}\rho(\delta x)$ where $\rho\in C^\infty(\RR^2)$ is a non-negative radial function with $\supp \rho\subset B_1(0)$ and $\int_{\RR^2}\rho=1$.

\begin{proposition}\label{propufacts}
The velocity field $u\in C_c^\infty$ defined above has the following properties:
\begin{enumerate}
\item $u(\sigma_1 x)=\sigma_1u(x)$ and $u(\sigma_2 x)=\sigma_2u(x)$, where $\sigma_1$, $\sigma_2$ are the reflections defined in the last section.
\item $u=(-1,0)$ on $R_{0}\cap H_+^{\delta}$, and $u=(1,0)$ on $R_{0}\cap H_-^{\delta}$.   
\item In particular, $\supp \nabla u\subset \overline{R_\varepsilon\backslash (R_0\cap (H_+^\delta\cup H_-^\delta))}$.
\item \[\|u\|_{L^\infty}\asymp \delta^{-1},\quad \|\nabla u\|_{L^\infty}\asymp\delta^{-2}\]
\end{enumerate}
\end{proposition}The proof of these facts is elementary.

We now introduce rescaled and rotated copies of $u$, corresponding to the images $F_w^t(I)$ for any finite word $w\in \{1,2\}^\ast$. 
For a vectorfield $v$, let
\begin{equation}\label{eqFTildeDefn1}
\widetilde{F}_1^t v(x)\coloneqq\mathcal{R}v\left(\alpha^{-1}\mathcal{R}^{-1}\left[x-\left(1-\int_0^t\eta,0\right)\right]\right)=\mathcal{R} v\circ (F_1^t)^{-1}
\end{equation}
and
\begin{equation}\label{eqFTildeDefn2}
\widetilde{F}_2^t v(x)\coloneqq\mathcal{R}v\left(\alpha^{-1}\mathcal{R}^{-1}\left[x+\left(1-\int_0^t\eta,0\right)\right]\right)=\mathcal{R} v\circ (F_2^t)^{-1}.
\end{equation}
Note that \begin{equation}\label{eqFandSuppTildeF}\supp (\widetilde{F}_i^t v)=F_i^t(\supp (v))\end{equation} and that $\widetilde{F}_i v$ are linear with respect to $v$. 

For any $k$ let $\widetilde{F}_w^t\coloneqq \widetilde{F}_{w_1}^t\circ\ldots\circ \widetilde{F}_{w_k}^t$ for any  $w\in \{1,2\}^k$, i.e.\ for $w\in\{1,2\}^k$
\begin{equation}\label{eqTildeFexpr}
\widetilde{F}_w^t v=\mathcal{R}^kv\circ (F_w^t)^{-1}.
\end{equation} It follows that 
\begin{equation}\label{eqTildeFXDerivSupp}\supp (\nabla \widetilde{F}_w^t v)\subset F_w^t(\supp\nabla v),
\end{equation} and 
\begin{equation}\label{eqTildeFTimeDerivSupp}\supp (\p_t\widetilde{F}_w^t v)\subset F_w^t(\supp\nabla v).\end{equation}
We also define $\widetilde{F}_\emptyset^tv=v$ for all $v:\RR^2\to\RR^2$.

Combining the observations above, we have deduce the following properties of $\widetilde{F}_wu$:
\begin{proposition}\label{factDisjointSuppts}
\begin{enumerate}\item $\supp(\nabla\widetilde{F}_w^tu)\cap\supp(\nabla\widetilde{F}_{w'}^tu)=\emptyset$ for $w,w'\in \{1,2\}^\ast$ with $w\neq w'$. 
\item $\supp(\p_t\widetilde{F}_w^tu)\cap\supp(\p_t\widetilde{F}_{w'}^tu)=\emptyset$ for $w,w'\in \{1,2\}^\ast$ with $w\neq w'$.
\item $\nabla\cdot \widetilde{F}_w u=0$.
\end{enumerate}
\end{proposition}
\begin{proof} The first two claims follow from Proposition \partref{propufacts}{.3}, Lemma \ref{factDisjointBorders}, \re{eqTildeFXDerivSupp}, and \re{eqTildeFTimeDerivSupp}.
The third claim follows from \re{eqTildeFexpr}, since for each $w$, $F^t_w(x) = c_{t,w}+\alpha^k\cR^kx$ for some $c_{t,w}\in\RR^2$.
\end{proof}
From Lemma \ref{factHalfSpace2}, and Proposition \partref{propufacts}{.2}, it follows that
\[
\supp(\widetilde{F}_1^t u)\subset R_0\cap H_+^\delta
\subset\{x\,\colon  u(x)= (-1,0)\},
\]and\[
\supp(\widetilde{F}_2^t u)\subset R_0\cap H_-^\delta
\subset\{x\,\colon u(x)=(1,0)\}.
\]
Moreover by \re{eqTildeFexpr}, for $k\in\NN$ and $w\in\{1,2\}^k$, 
\begin{multline*}
\supp(\widetilde{F}_w^t\circ \widetilde{F}_1^t u)= F_w^t(\supp(\widetilde{F}_1^t u))\subset F_w^t(R_0\cap H_+^\delta)\\
\subset\{x\,\colon \widetilde{F}_w^t u(x)=\mathcal{R}^k (-1,0)\},
\end{multline*} and
\[
\supp(\widetilde{F}_w^t\circ \widetilde{F}_2^t u)\subset\{x\,\colon \widetilde{F}_w^t u(x)=\mathcal{R}^k (1,0)\}.
\] 

We can now define a velocity that contracts $S_\alpha$ by the factor $\gamma=1-\delta$ in unit time.
\begin{notation}\label{defnU}
For $t\in\RR$ let
\begin{equation}\label{defnU:eq1}
U(t,x)=U_\alpha(t,x)\coloneqq\eta(t)\sum_{k=0}^\infty \sum_{w\in\{1,2\}^k} \alpha^k\widetilde{F}_w^t u(x).
\end{equation}
\end{notation}
Since $\widetilde{F}_w^t u$ is bounded, this sum converges uniformly for all $t$, in particular, $U(t)$ is weakly divergence free. From the definition of $\widetilde{F}_w^t$ (\ref{eqFTildeDefn1}, \ref{eqFTildeDefn2}, and \ref{eqTildeFexpr}) we see that there exists $C_\eta>0$, independent of $w\in\{1,2\}^\ast$, with $C_\eta\lesssim\|\eta\|_{L^\infty}$ such that
\[
\|\p_t\widetilde{F}_w^t u\|_{L^\infty_{x,t}} \leq C_\eta \frac{\alpha^{-k}-1}{1-\alpha}\|\nabla u\|_{L^\infty} ,
\]
and
\[
\|\nabla\widetilde{F}_w^t u\|_{L^\infty_{x,t}} \leq \alpha^{-k}\|\nabla u\|_{L^\infty} . 
\]
Furthermore, by Propositions \partref{factDisjointSuppts}{.1} and \partref{factDisjointSuppts}{.2}, the supports of derivatives of the summands in \re{defnU:eq1} are pairwise disjoint. In combination with the uniform bounds above, this implies that each partial sum in \re{defnU:eq1} is Lipschitz, with constant bounded independent of $k$. Thus, the uniform limit  $U$ is Lipschitz in spacetime with 
\begin{equation}\label{eqULipConstT}
\|\p_t U\|_{L^\infty_{t,x}}\leq C_{\eta,\p_t\eta}\left(\|u\|_{L^\infty}+(1-\alpha)^{-1} \|\nabla u\|_{L^\infty}\right)
\end{equation} 
and
\begin{equation}\label{eqULipConstX}
\|\nabla U\|_{L^\infty_{t,x}}\leq \|\eta\|_{L^\infty} \|\nabla u\|_{L^\infty}\leq C \delta^{-1},
\end{equation} for some $C_{\eta,\p_t\eta}\lesssim \delta,\ C>0$.

By Proposition \ref{propufacts}, we obtain an $\alpha$-independent estimate, which will be useful later
\[
\|U_\alpha\|_{L^\infty_{t,x}}\leq \|\eta\|_{L^\infty}\sum_{k=0}^\infty 2^{-k/2}\|u\|_{L^\infty}
\leq C.
\]
for some $C>0$ independent of $\alpha$.

\subsection{A flow contracting \texorpdfstring{$S_\alpha$}{Sɑ} to \texorpdfstring{$0$}{0}}

\subsubsection{Contraction of \texorpdfstring{$S_\alpha$}{Sɑ} due to \texorpdfstring{$U$}{U}}
Since $U$ is uniformly Lipschitz, its trajectory map $X_U$ is well defined  and continuous, hence if we can show that $X_U(1,F_w^0(I))=\gamma F_w^0(I)$ for all $w\in\{1,2\}^\ast$ then by \re{eqFracCharacter} it will follow that $X_U(1,S_\alpha)=\gamma S_\alpha$. Recall $\gamma=\gamma_\alpha$ is specified in Notation \ref{notConstants}.

The following lemmas allow us to express  $X_U$, restricted to $S_\alpha$, in terms of only the terms in \re{defnU:eq1} with $k=0,1$, and a dilation.
\begin{lemma}\label{lemUonLines} For any $k\in\NN_0$, $w\in \{1,2\}^k$, $y\in I$ and $t\in[0,1]$, 
\[
U(t,F_w^t(y))=\eta(t)\left(\alpha^k\mathcal{R}^k[u+\alpha \widetilde{F}_1^t u+\alpha\widetilde{F}_2^tu](y) - F_w^0(0) \right)
\]
\end{lemma}
\begin{proof}
We see that for all $t\in[0,1]$, if $w'\in \{1,2\}^\ell$ with $\ell\geq k+2$, then
\[\supp(\widetilde{F}_{w'}^tu)\cap F_w^t(I)= F_{w'}^t(\supp( u))\cap F_w^t(I)=\emptyset.
\]  Indeed if $w'_i=w_i$ for $i=1,\ldots, k$ this follows from the injectivity of $F^t_w$ and the fact that $I\cap F_v^t(R_\varepsilon)=\emptyset$ for any $v\in\{1,2\}^2$. If, on the other hand,  $w_i\neq w'_i$ for some minimal $i\leq k$, we have \[\left(F_{w'}^t(R_\varepsilon)\cap F_w^t(I)\right)\subset\left( F_{(w_1,\ldots,w_i)}^t(R_\varepsilon)\cap F_{(w_1,\ldots,w_{i-1},w_i')}^t(R_\varepsilon)\right)=\emptyset,
\]by Lemma  \ref{factHalfSpace2}

It follows that for $y\in I$, $U(t,F_w^t(y))$ is given by
\begin{multline*}
U(t,F_w^t(y))
=\eta(t)\sum_{\ell=0}^{k-1}\alpha^\ell \widetilde{F}_{(w_1,\ldots, w_\ell)}^tu(F_w^t(y)) + \eta(t)\alpha^{k}\widetilde{F}_w^tu\big|_{F_w^t(y)}\\+ \eta(t)\alpha^{k+1}\widetilde{F}_{w}^t\left(\widetilde{F}_{1}^tu\right)\big|_{F_w^t(y)}+ \eta(t)\alpha^{k+1}\widetilde{F}_{w}^t\left(\widetilde{F}_{2}^tu\right)\big|_{F_w^t(y)}\\
=\eta(t)\sum_{\ell=0}^{k-1}\alpha^\ell (-1)^{w_{\ell+1}}\mathcal{R}^\ell(1,0) + \eta(t)\alpha^{k}\mathcal{R}^k\left(u+\alpha\widetilde{F}_1^tu+\alpha\widetilde{F}_2^tu\right)(y),
\end{multline*}
where we have used Lemma \ref{factHalfSpace2} and Proposition \partref{propufacts}{.2} to see that $u\big |_{F_i(R_\varepsilon)}=(-1)^i(1,0)$.

It remains to check that
\[
-\sum_{\ell=0}^{k-1}\alpha^\ell (-1)^{w_{\ell+1}}\mathcal{R}^\ell(1,0)=F_w^0(0).
\] But this is indeed the case, by a simple inductive argument.
\end{proof}

Similarly, we have
\begin{lemma}\label{lemBoxTranslations}
For $y\in R_\varepsilon$, and $w\in \{1,2\}^\ast$,
\[
F_w^t(y)=F_w^0(y)-F_w^0(0)\int_0^t\eta(s)\,\d s. 
\]
\end{lemma}
\begin{proof}
For $w\in\{\emptyset,(1),(2)\}$ this follows directly from the definition of $F_w^t$. Suppose by induction that the identity holds for $w'\in \{1,2\}^k$, we check that it is also true for $w=(1,w'_1,\ldots,w'_k)$, the other case being similar. In this case, by definition and hypothesis we have,
\begin{multline*}
F_w^t(y)=\left(1-\int_0^t\eta,0\right)+\alpha \mathcal{R}(F_{w'}^t(y)) \\
= \left(1-\int_0^t\eta,0\right)+\alpha \mathcal{R}\left(F_{w'}^0(y)-F_{w'}^0(0)\int_0^t\eta\right) \\
=F_w^0(y)-F_w^0(0)\int_0^t\eta. 
\end{multline*}
\end{proof}

\begin{lemma}\label{lemY}
The trajectory map associated to the velocity
\[
v\coloneqq\eta(t)[u+\alpha \widetilde{F}_1^t u +\alpha \widetilde{F}_2^t u]
\]
satisfies $X_v(t,I)=\left[X_v(t,(-1,0)),X_v(t,(1,0))\right]\x\{0\}\subset I$ for $t\in \RR$.
\end{lemma}
\begin{proof}
Since $u\in C_c^\infty$, $v$ admits a well-defined, continuous trajectory map $X_v$. In particular, by uniqueness of trajectories, it suffices to check that $X_v(t,I)\subset \RR\x\{0\}$ and calculate $X_v(t,(\pm 1,0))$.

By the reflective symmetries in Proposition \partref{propufacts}{.1} we see that $u(x_1,0)=(u_1(x_1,0),0)$ and $u(0,x_2)=(0,u_2(0,x_2))$. Now as $(F_i^t)^{-1}(\RR\x\{0\})\subset \{0\}\x\RR$, we have 
\[\widetilde{F}_i^t u(x_1,0)=\mathcal{R} u ((F_i^t)^{-1}(x_1,0))\in\mathcal{R}u(\{0\}\x \RR)\subset\RR\x\{0\}.
\]
Hence indeed $v(\RR\x\{0\},t)\subset\RR\x\{0\}$ for $t\in[0,1]$, thus $X_v$ preserves $\RR\x\{0\}$.

It now suffices to check that
\[
X_v(t,(\pm 1,0))=\pm \left(1-\int_0^t\eta,0\right)\eqqcolon \pm r(t)\in I,
\]
but indeed $\widetilde{F}_i^t u(r(t))=0$ for $i=1,2$, so
\[
v(\pm r(t)) = \pm \eta(t)(-1,0) = \pm\frac{\d }{\d t}r(t),\quad r(0)=\pm(1,0).
\]
Hence $\pm r(t)=X_v(t,\pm(1,0))$ as claimed. 
\end{proof}

The above lemmas allow us to calculate $X_U(t,F_w^0(I))$ for all $w\in \{1,2\}^\ast$.
\begin{lemma}\label{lemXI}
For any $w\in \{1,2\}^\ast$, $X_v$ as defined in Lemma \ref{lemY}, and $y\in I$
\[
X_U(t,F_w^0(y))=F_w^t\circ X_v(t,y).
\] 
In particular
\begin{equation}\label{lemXI:eq2}
X_U(1,F_w^0(I))=\gamma F_w^0(I).
\end{equation}
\end{lemma}
\begin{proof}
The second claim follows from the first by Lemmas \ref{lemBoxTranslations} and \ref{lemY}. Indeed,
\begin{multline*}
F_w^1\circ X_v(1,I)=F_w^0(X_v(1,I))-F_w^0(0)\int_0^1\eta\\
 =F_w^0([\delta-1,1-\delta]\x\{0\})-F_w^0(0)\delta
=(1-\delta)F_w^0(I)=\gamma F_w^0(I),
\end{multline*}
since $F_w^0$ is affine linear.

To prove the main claim, fix $y\in I$ and $w\in\{1,2\}^k$. By Lemma \ref{lemBoxTranslations}, we have
\[
F_w^t\circ X_v(t,y) =F_w^0\circ X_v(t,y)-F_w^0(0)\int_0^t\eta.
\] 
Thus, (using the fact that $F_w^0$ is affine linear),
\begin{multline*}
\frac{\d}{\d t}\left(F_w^t\circ X_v(t,y) \right)=F_w^0(\p_t X_v(t,y))-F_w^0(0) - F_w^0(0)\eta(t)\\
=\alpha^k\mathcal{R}^k\eta(t)\left[u+\alpha\widetilde{F}_1^t u+\alpha\widetilde{F}_2^t u\right]_{X_v(t,y)} - F_w^0(0)\eta(t)
\end{multline*}
By Lemma \ref{lemY}, $y\in I$ implies that $X_v(y,t)\in I$, hence applying Lemma \ref{lemUonLines} yields
\[
\frac{\d}{\d t}\left(F_w^t\circ X_v(t,y) \right)=U(t,F_w^t(X_v(t,y))),\quad F_w^0\circ X_v(0,y)=F_w^0(y).
\]
The result follows by uniqueness of the trajectory $X_U(\cdot,F_w^0(y))$.
\end{proof}

As remarked at the beginning of the section, by \re{lemXI:eq2} we have proved the following proposition.

\begin{proposition}\label{propU}
\[
X_U(1,S_\alpha)=\gamma S_\alpha.
\]
\end{proposition}

\subsubsection*{A contraction to $0$}
We now define a velocity that flows $S_\alpha$ to the origin in finite time. Let $\xi\in(\sqrt{\gamma},1)$ then define, for $t\in\RR$,
\begin{equation}\label{eqW}
W(x,t)\coloneqq\sum_{k=0}^\infty \left(\frac{\gamma}{\xi}\right)^{k}U\left(\frac{t-t_k}{\xi^k},\frac{x}{\gamma^k}\right)
\end{equation}
where $t_0=0$, and $t_k\coloneqq\sum_{j=0}^{k-1}\xi^j<\frac{1}{1-\xi}$ for $k>0$. For the summands in \re{eqW} we use the notation \begin{equation}\label{eqWk}
W_k\coloneqq \left(\frac{\gamma}{\xi}\right)^{k}U\left(\frac{t-t_k}{\xi^k},\frac{x}{\gamma^k}\right)
\end{equation}
The following lemma completes the proof of Theorem \ref{thmFractal1}.
\begin{lemma}
For any $x\in S_\alpha$, $X_W(t,x)\to 0$ as $t\to\frac{1}{1-\xi}$ where $X_W$ is the trajectory map corresponding to $W$.
\end{lemma}
\begin{proof}
The trajectory map corresponding to $W_k$ is
\[
X_{W_k}(t,x)=\gamma^kX_U\left(\frac{t-t_k}{\xi^k},\frac{x}{\gamma^k}\right),
\] 
by a simple time rescaling argument, and Lemma \ref{factXrescale}. It follows inductively that
\[
X_W(t_{k},S_\alpha)=\gamma^k S_\alpha\subset \gamma^k R_\varepsilon.
\]
For $t\in(t_k,t_{k+1})$ observe that $\supp W_k(t,\cdot)\subset\gamma^k R_\varepsilon$. Combining these facts yields
\[
X_W(t,S_\alpha)\subset\gamma^k R_\varepsilon,\quad \mbox{ for }\quad t\in[t_k,t_{k+1}],
\] 
so $X_W(t,x)\to 0$ as $t\to\frac{1}{1-\xi}$ uniformly with respect to $x\in S_\alpha$.
\end{proof}

Finally, we verify the Sobolev regularity of the velocity $W$.
\begin{proposition}\label{propWregularity}
The map $W$ is uniformly continuous on $[0,\frac{1}{1-\xi}]\x\RR^2$ and uniformly Lipschitz on $[0,T]\x\RR^2$ for any $T\in(0,\frac{1}{1-\xi})$. Moreover $W\in L^q(0,\frac{1}{1-\xi};W^{1,p})$ for any $p\in[1,\infty)$, $1\leq q\leq\infty$ such that
\begin{equation}\label{propWregularity:eqPQ}
p(1-1/q)<\frac{\log\gamma}{\log\xi}
\end{equation}
\end{proposition}
\begin{proof}
Note that the temporal supports of the summands in \re{eqW} are mutually disjoint, and $W_k$
is, for each $t$, Lipschitz with respect to $x$ with constant $\xi^{-k}L$, where $L$ is the space-time Lipschitz constant of $U$. Also for any $x$, $W_k(t,x)$ is Lipschitz with respect to $t$ with constant $\gamma^k\xi^{-2k}L<L$.
Hence, for any $T\in(0,\frac{1}{1-\xi})$ $U$ is  Lipschitz on $[0,T]\x\RR^2$. 

Since $U$ is bounded, there exists $C>0$ such that $\|W(\cdot,t)\|_{L^\infty}\leq C\left(\frac{\gamma}{\xi}\right)^k$ for 
\[t\in\left[\frac{1-\xi^{k-1}}{1-\xi},\frac{1-\xi^{k}}{1-\xi}\right].\] Combined with the Lipschitz properties on $[0,\frac{1}{1-\xi})$, this uniform convergence implies that $W$ is uniformly continuous in $[0,\frac{1}{1-\xi}]\x\RR^2$.   

For the Sobolev regularity, it suffices to estimate $\nabla W$, as $W\in L^\infty(0,\frac{1}{1-\xi}; L^\infty)$ and has compact support. Now by \re{eqULipConstX}, for $t\in[t_k,t_k+\xi^k]$ and $1\leq p<\infty$
\[
\|\nabla W(t)\|_{L^p}\lesssim\xi^{-k}\gamma^{k/p}\delta^{-1},
\] 
hence $W\in L^q(0,\frac{1}{1-\xi};W^{1,p})$ if
\[
\xi^{-1+1/q}\gamma^{1/p}<1.
\]
which is satisfied if \re{propWregularity:eqPQ} holds.
\end{proof}
\subsection{A non-uniqueness of full dimension}\label{secPassiveScalar2}

We adapt the construction above for a sequence $\alpha_n\to \frac{1}{\sqrt{2}}$, to exhibit  a divergence-free vectorfield flowing a two-dimensional set (of measure zero) into the origin in finite time. For the sake of clarity we will choose a number of explicit constants during the construction, these choices do not affect the strength of the result.

Theorem \ref{thmPassiveScalar2} is a consequence of the following theorem, with the applications of Lemmas  \ref{lemPassToWeak2D} and \ref{lemReversibility} to the restricted Hausdorff measure $\mathcal{H}^2_S$.

\begin{theorem}\label{thmFractal2}
There exists a uniformly continuous, weakly divergence-free velocity $\widetilde{V}:[0,T]\x\RR^2\to\RR^2$ for some $T>0$,  and a compact set $S\subset\RR^2$ with Hausdorff dimension 2 for which there exists a corresponding Lagrangian flow $X_{\widetilde{V}}$ satisfying
\[
\dim_HX_{\widetilde{V}}(t,S)=2\quad \forall t\in[0,T),\qquad X_{\widetilde{V}}(T,S)=\{0\}.
\]
For all $t\in[0,T]$, $\widetilde{V}(t,\cdot)$ is locally Lipschitz on $\RR^2\backslash\{(0,0)\}$.
\end{theorem}

\subsubsection{An auxiliary flow \texorpdfstring{$\nu$}{ν} }
The velocities $U_\alpha$ from the previous section are the key building block in the proof of Theorem \ref{thmFractal2}, but in order to flow $S$ to a single point, we additionally construct a velocity that advects certain affine copies of sets $S_{\alpha_n}$.

\begin{notation}
Let $\tilde\chi\in C^\infty(\RR^2;[0,1])$ be such that
\[
\supp (\tilde\chi)\subset H_+,\quad \tilde\chi\equiv 1 \mbox{ on } H^{3/4}_+,
\] having bounded first derivatives $\|D\tilde\chi\|_{L^\infty}\leq C$ and symmetry in the $x$-axis: $\tilde\chi(\sigma_1 x)=\tilde\chi(x)$.
\end{notation}

\begin{notation}
For $k\in\NN_0$ and $t\in[0,1]$, let 
\[\chi_k(t,x)\coloneqq (7/8)^{k}\tilde\chi\left(\left(8/7\right)^k [x-g_k(t)(1,0)]\right),
\]
where
\[
g_k(t)\coloneqq \left(\frac{7}{8}\right)^{k+1}\left[25-3t\right].
\]
\end{notation}
Observe that for $k\in\NN$
\begin{equation}\label{eqSuppChiFact1}\supp\chi_k(t,\cdot)\subset H_+^{g_k(t)},\end{equation} and, for $t\in(0,1]$,
\begin{equation}\label{eqSuppChiFact2}
g_{k+1}(0)<g_{k}(t)<g_{k}(0).
\end{equation}
\begin{proposition}\label{factChiSupp}
For all $t\in[0,1]$
\[
\supp(\nabla\chi_k(t,\cdot))\cap \supp(\nabla\chi_j(t,\cdot))=\emptyset
\]
if $j\neq k$.
\end{proposition}
\begin{proof}
This follows from the fact that
 \begin{multline}\label{factChiSupp:eq1}
\supp(\nabla\chi_k(t,\cdot))\subset  [g_k(t),g_k(t)+(7/8)^k3/4]\x\RR\\
\subset [g_k(1),(7/8)^k(g_0(0)+3/4)]\x\RR,
\end{multline}
and the right-hand side is contained in $[g_k(1),g_{k-1}(1))\x\RR$ for $k\geq 1$.
\end{proof}
Similarly, we see that $\p_t\chi_k$ and $\p_t\chi_j$ have disjoint support if $j\neq k$. 

\begin{notation}
Let  $\Phi:\RR^2\to[0,1]$ be smooth in $(0,\infty)\x\RR$  such that
\[
\supp(\Phi)\subset \{x\,\colon x_1>|x_2|\}\cup\{(0,0)\},\ \Phi\equiv 1 \mbox{ on } [0,25]\x\RR \cap \{x\,\colon x_1>10|x_2|\},
\]
and $|D^\ell\Phi(x)|\lesssim |x|^{-\ell}$ for $\ell=1,2$.
\end{notation}

\begin{notation}
For $k\in\NN_0$ let $\nu_k\in C^\infty_c([0,1]\x\RR^2;\RR^2)$ be the divergence-free vectorfield given by
\[
\nu_k(t,x)\coloneqq\nabla^\perp(x_2\Phi(x)\chi_k(t,x)).
\]
Now define a vectorfield $\nu$ for $t\in[0,1]$ by
\begin{equation}\label{eqOmega}
\nu(t,x)\coloneqq\tilde\eta(t)\frac{3}{8}\sum_{k=0}^\infty \nu_k\left(\int_0^t\tilde \eta,x\right),
\end{equation}
where $\tilde\eta\in C^\infty_c((0,1);\RR)$ is a non-negative function such that $\int_0^1\tilde\eta=1$.
\end{notation}

\begin{proposition}
The sum in \re{eqOmega} converges uniformly and $\nu$ is uniformly Lipschitz in $[0,1]\x\RR^2$.
\end{proposition}
\begin{proof}
Uniform convergence follows from the estimate
\begin{multline*}
\|\nu_k\|_{L^\infty} \lesssim (\|\Phi\|_{L^\infty}+\||x|\nabla\Phi\|_{L^\infty(\supp(\chi_k))})\|\chi_k\|_{L^\infty} \\+ \|x\|_{L^\infty(\supp(\Phi)\cap\supp(\nabla\chi_k))}
\lesssim \left(\frac{7}{8}\right)^k,
\end{multline*} where we have used \re{factChiSupp:eq1}.

For the Lipschitz property, it follows from Proposition \ref{factChiSupp} that for any $\ell\in\NN$ and any first-order spatial derivative $\p_x$,
\begin{multline}\label{eqOmegaConv1}
\left|\p_x\sum_{k=0}^\ell\nu_k(t,x)\right| \leq \left|\sum_{k=0}^\ell\left(\p_x\nabla^\perp [x_2\Phi(x)]\right)\chi_k(t,x)\right| \\+ 2|\nabla[x_2\Phi(x)]||\nabla\chi_j(t,x)| + |x_2\Phi(x)||D^2\chi_j(t,x)| 
\end{multline}
for all $x$ and for some $j=j(x)\in 0,1,\ldots,\ell$.

 For $x\in [g_m(1),g_{m-1}(0)]\x\RR$ and $t\in[0,1]$, we use \re{eqSuppChiFact1} and \re{eqSuppChiFact2}
  to estimate the first term as follows:
\[
  \left|\sum_{k=0}^\ell(\p_x\nabla^\perp [x_2\Phi(x)])\chi_k\right|\lesssim|x|^{-1}\sum_{k=0}^\ell|\chi_k(t,x)|\lesssim \left(\frac{8}{7}\right)^m \sum_{k=m-1}^\infty \left(\frac{7}{8}\right)^k\leq C,
\]
for some $C>0$ independent of $m$, $\ell$ and $t$. Since $g_m(1)<g_m(0)$ for all $m\in\NN$ and $g_m(0)\to 0$ as $m\to\infty$, this term gives a uniformly bounded contribution to the derivative of \re{eqOmega} for all $x\in (0,\infty)\x\RR$.

The second term in \re{eqOmegaConv1}, is uniformly bounded in $[0,1]\x \RR^2$ since $|x||\nabla\Phi|$, $|\nabla\chi_j|$ and $|\Phi|$ are. For the last term we have,
\[
|x_2\Phi||D^2\chi_j|\lesssim g_0(0)+3/4
\]
since $\supp(\nabla\chi_j)\cap\supp(\Phi)\subset\{x\,\colon x_2\leq(7/8)^k(g_0(0)+3/4)\}$, by \re{factChiSupp:eq1} and the definition of $\Phi$.

A similar argument shows that $\p_t\sum_{k=1}^\ell\nu_k$ is uniformly bounded (with respect to $t\in[0,1]$, $x\in(0,\infty)\x\RR$, and $\ell\in\NN$).

It follows that the partial sum of \re{eqOmega} of order $\ell$ is Lipschitz in $[0,1]\x (0,\infty)\x\RR$, with constant independent of $\ell$.
By \re{eqSuppChiFact1}, the partial sums also have support contained in $[0,1]\x (0,\infty)\x\RR$, and are therefore also uniformly Lipschitz in $[0,1]\x\RR^2$. 
We conclude that $\nu$ is Lipschitz in $[0,1]\x \RR^2$ as the uniform limit of these partial sums.  
\end{proof}

By the choice of $\tilde\eta$, $\nu$ can be extended (by zero) to a Lipschitz function on $[0,\infty)\x \RR^2$. Moreover, $\nu$ is weakly divergence free, as the uniform limit of divergence-free functions.

The following property is the reason for constructing $\nu$.
\begin{proposition}\label{propOmegaTraj}
Let $Q\coloneqq [-1,1]\x[-\sqrt{2},\sqrt{2}]$, and let $X_\nu$ denote the trajectory map associated to $\nu$, then
\begin{equation}\label{propOmegaTraj:eq1}
X_\nu\left(1,(7/8)^k\left[x+(24,0)\right]\right)=(7/8)^{k}[x+(7/8)(24,0)]
\end{equation}
for all $k\in\NN_0$ and $x\in Q$.
\end{proposition}
\begin{proof}
To simplify notation we may reparametrise in time, and assume that $\tilde\eta\equiv 1$. Indeed, consider the velocity $\hat\nu$ given by \re{eqOmega}, with $\eta\equiv 1$. If $\hat y(t)$ is a trajectory of $\hat\nu$, i.e.\ $\p_t \hat y=\hat\nu(t,y(t))$, then $y(t)=\hat y(\int_0^t\eta)$ is a trajectory of $\nu$ and $y(1)=\hat y(1)$. 

After this simplification, it is enough to check that  
\[
\nu(t,y)=-24(7/8)^k(1/8,0)=-(7/8)^k(3,0),
\] for 
\begin{equation}\label{propOmegaTraj:eq_y}y\in (7/8)^k[Q+24(1-t/8,0)].
\end{equation} Indeed, if this is the case, and if $y_0\in (7/8)^k[Q+(24,0)]$, then $y(t)=y_0-\frac{24 t}{8}(\frac{7}{8})^k(1,0)$ is a trajectory for $\nu$, i.e.\ $\p_t y=\nu(t,y(t))$, and $y(1)=y_0-\frac{24}{8}(\frac{7}{8})^k(1,0)$, which can be written in the form required in \re{propOmegaTraj:eq1}.

For $y$ as in \re{propOmegaTraj:eq_y},  we have 
\begin{multline*}
y_1\geq \left(\frac{7}{8}\right)^k\left[24(1-t/8)-1\right]=g_k(t)+ \left(\frac{7}{8}\right)^k \left(\frac{9}{8}-\frac{3}{8}t\right)\\
\geq g_k(t)+ \left(\frac{7}{8}\right)^k \frac{3}{4},
\end{multline*}
(recall that $g_k=(7/8)^{k+1}[25-3t]$) and
\begin{equation}\label{eqY1UpperBd}
y_1 \leq \left(\frac{7}{8}\right)^k\left[24(1-t/8)+1\right]=\left(\frac{7}{8}\right)^k\left[25-3t\right]= g_{k-1}(t).
\end{equation}
Moreover 
\[
|y_2|\leq \left(\frac{7}{8}\right)^k\sqrt{2} <(7/8)^k2=\frac{1}{10}\left[g_k(1)+ \left(\frac{7}{8}\right)^k \left(\frac{3}{4}\right)\right]\leq \frac{y_1}{10} .
\]
Since $\chi_j(t)\equiv (7/8)^j$ on $H_+^{g_j(t)+(7/8)^j(3/4)}$, and $\Phi\equiv 1$ on $\{x\,\colon x_1>10|x_2|\}$, we see that 
\[
\nu_j(t,y)=\left\{\begin{array}{ll}
0&j<k,\\
-(7/8)^j (1,0)& j\geq k.
\end{array}\right.
\]
To verify the case $j<k$, \re{eqY1UpperBd} implies that $y\notin \INT({H_+^{g_{k-1}(t)}})$ and $\supp\nu_j(t,\cdot)\subset H_+^{g_j(t)}$ so indeed, $\nu_j(t,y)=0$.
Hence (summing over $j$, and subject to reparametrising in time) $\nu(t,y)=-\left(\frac{7}{8}\right)^k(3,0)$, as required.
\end{proof}
\begin{corollary}\label{corLineTraj}
\[
X_\nu(1,[0,24]\x\{0\})=[0,24\cdot 7/8]\x\{0\}.
\]
\end{corollary}
\begin{proof}
By symmetry of $\chi_k$, $X_\nu$ preserves $\RR\x\{0\}$. The result follows from the fact that $\nu(0)=0$ and $\nu(t,(24-3t,0))=(-3,0)$, as above, so $X_\nu(1,(24,0))=(21,0)=(7/8)(24,0).$  
\end{proof}

\subsubsection{Combining \texorpdfstring{$\nu$}{ν} and \texorpdfstring{$U_\alpha$}{Uɑ}}

Continuing the proof of Theorem \ref{thmFractal2}, for a particular sequence $(\alpha_n)_{n=1}^\infty\subset(\frac{1}{2},\frac{1}{\sqrt{2}})$ with $\alpha_n\to\frac{1}{\sqrt{2}}$,  we combine rescaled and rotated copies of $U_{\alpha_n}$ and $\nu$ to construct a uniformly continuous divergence-free velocity field that is locally Lipschitz away from the origin and flows a two-dimensional set to the origin in finite time.

\begin{notation}
For $n\geq 1$, define
\[
\alpha_n\coloneqq 2\sqrt{2}\left(\left(\frac{7}{8}\right)^{1/(n+1)}-\frac{3}{4}\right).
\]
\end{notation}
The contraction ratio of $S_{\alpha_n}$ (defined by \re{eqSAlphaDefn}) under the flow $U_{\alpha_n}$ (defined in  \re{defnU:eq1}) is then \[\gamma_n \coloneqq \gamma_{\alpha_n}=(7/8)^{1/(n+1)}.\] 
Next define a collection of affine linear maps:
\[
G_n(x)=\left(\frac{7}{8}\right)^{n-1}\left[\frac{1}{\sqrt{2}}\mathcal{R}x + (24,0)\right].
\]
Then denote by $S$ the set
\[
S\coloneqq ([0,24]\x\{0\}) \cup \bigcup_{k=1}^\infty G_{k}( S_{\alpha_k}),
\]
(see Figure \ref{fig:S2D}). As the union of sets with dimensions 
\begin{equation}\label{eqLimDimSAlphaK}\max(1,-\log(2)/\log(\alpha_k))\to 2\end{equation} as $k\to\infty$, it follows that $\dim_H S=2$.

\begin{figure}
\centering
\includegraphics[width=0.7\textwidth]{\FiguresPath 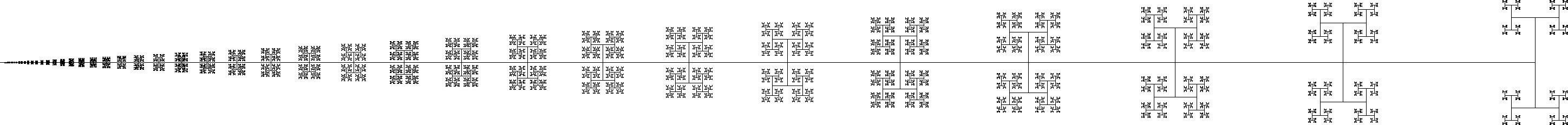}
\caption{$\dim_HS=2$}
\label{fig:S2D}
\end{figure}

We can now present the velocity that contracts $S$ by the factor $7/8$.
\begin{equation}\label{eqVDefn}
V(t,x)\coloneqq\left\{\begin{array}{ll}
0&t<0\\
\sum_{k=1}^\infty \left(\frac{7}{8}\right)^{k-1}\mathcal{R}V_k(t,G_k^{-1}(x))&t\in[0,1]
\\
\nu(t-1,x)& t\in[1,2] \\
0 &t>2
\end{array}
\right.
\end{equation}
where $V_k$ is the following finite sum of rescaled copies of $U_{\alpha_k}$ which induces a contraction of $S_{\alpha_k}$ by a factor $7/8$ in unit time, independent of $k$: 
\begin{equation}\label{eqVkDefn}
V_k(t,x)\coloneqq(k+1)\sum_{n=0}^k\gamma_k^nU_{\alpha_k}\left((k+1)t-n,\gamma_k^{-n}x\right).
\end{equation}

Since $X_{U_{\alpha_k}}(1,S_{\alpha_k})=\gamma_k S_{\alpha_k}$, we see that for $n=0,1,\ldots,k+1$
\begin{equation}\label{eqXVkContra}
X_{V_k}\left(\frac{n}{k+1},S_{\alpha_k}\right)=\gamma_k^n S_{\alpha_k}.
\end{equation}
In particular $X_{V_k}(1,S_{\alpha_k})=\gamma_k^{k+1}S_{\alpha_k}=\frac{7}{8}S_{\alpha_k}$. We also have, by symmetry of $U_{\alpha_k}$ that 
\[X_{V_k}\left(1,(\{0\}\x[0,\pm\infty)) \cap R_{1/2}\right)=(\{0\}\x[0,\pm\infty) )\cap R_{1/2},\]
so, using Lemma \ref{factXrescale},
\[
X_{\mathcal{R}V_k\circ G_k^{-1}}(1,G_k(R_{1/2})\cap\left([0,\infty)\x \{0\})\right) = G_k(R_{1/2})\cap([0,\infty)\x \{0\}).
\]

Considering the supports of the summands in  \re{eqVkDefn}, $\gamma_k^nG_k(\supp U_{\alpha_k})\subset G_k(\supp U_{\alpha_k})\subset G_k(R_{1/2})$, and
\begin{multline}\label{eqSuppVk}
G_k(R_{1/2})\subset G_k(R_{1/\sqrt{2}}) \subset\left\{x\,\colon 24-\frac{3}{2}\leq\left(\frac{8}{7}\right)^{k-1} x_1\leq24+\frac{3}{2}\right\}\\
\subset \left\{x\,\colon 24-\frac{3}{2}\leq\left(\frac{8}{7}\right)^{k-1} x_1<\frac{8}{7}\left(24-\frac{3}{2}\right)\right\}.
\end{multline}
Combining this with the the definition of $G_k$ yields: 
\[
G_k(R_{1/2})\subset\left\{x\,\colon|x_2|\leq\left(\frac{7}{8}\right)^{k-1}\frac{2+1/2}{\sqrt{2}}<\left(\frac{7}{8}\right)^{k-1}\frac{45}{20}\leq \frac{x_1}{10}\right\}.
\]
Since $\supp \nu(t,\cdot)\subset\{x\,\colon |x_2|\leq x_1\}$ we conclude that
\begin{equation}\label{eqSuppV1}
\supp V\subset \RR\x \{x\,\colon |x_2|\leq x_1\},
\end{equation}
we will use this fact later.

A particular consequence of \re{eqSuppVk} is that $G_j(R_{1/2})\cap G_k(R_{1/2})=\emptyset$ if $j\neq k$, so the terms in the sum  in \re{eqVDefn} have disjoint support in space. Therefore we have:
\begin{proposition}
$V(t)$ is weakly divergence free for all $t\in\RR$.
\end{proposition}

\begin{lemma}
There exists a unique trajectory map $X_V$ associated to the velocity $V$.
\end{lemma}
\begin{proof}
For $t\in[1,2]$,  $V(t)=\nu(t-1)$, is Lipschitz (with constant independent of $t$), and $V$ vanishes outside of $t\in[0,2]$. It therefore suffices to check that the trajectories for the flow $V$ are unique for $t\in[0,1]$.
 
Note that $V$ is continuous in space-time and by \re{eqULipConstT},\re{eqULipConstX}, the first derivatives of $V_{k}$ satisfy
\[
\|\nabla V_k\|_{L^\infty}\lesssim k\delta_{\alpha_k}^{-1}.
\]
Since $\delta_{\alpha_k}\to 0$, $V$ may not be uniformly Lipschitz for $t\in[0,1]$, however
\begin{equation}\label{eqSuppV}\supp V(t)\subset \{0\}\cup\bigcup_{k=1}^\infty G_k(R_{1/2}) ,
\end{equation}
and $V(t)$ is Lipschitz (with constant independent of $t$) on $G_k(R_{1/2})$ for $k\in\NN_0$. Indeed, the union in \re{eqSuppV} is a disjoint union of compact sets by \re{eqSuppVk}, and $V(t)\big|_{G_k(R_{1/2})}=V_k$. Hence $X_V$ is well defined, subject to checking that the origin admits a unique trajectory. 

Indeed, if $Y\in C^1([0,1];\RR^2)$ with 
\[
\frac{\d}{\d t}Y(t)=V(t,Y(t)),\quad Y(0)=0,
\]
 then if $Y(t)\neq 0$, there exists $t_0\in(0,t]$ such that $Y(t_0)\in \supp V_k(t_0)$ for some $k$. This would imply that $Y(s)\in G_k(R_{1/2})$ for $s\in[t_0,1]$, by uniqueness of trajectories in that domain. We may assume that $t_0$ is minimal with respect to the condition $Y(t_0)\in G_k(R_{1/2})$. Now by a similar argument, there is $t_1\in(0,t_0)$ and $j\neq k$ (by minimality) such that $Y(s)\in G_j(R_{1/2})$ for $s\in[t_1,1]$, but this is a contradiction, since it implies $Y(t_0)\in G_k(R_{1/2})\cap G_j(R_{1/2})=\emptyset$.
\end{proof}

\begin{proposition}\label{thmXVscale}
\[
X_{V}(2,S)=\frac{7}{8}S.
\]
\end{proposition}
\begin{proof}
By \re{eqSuppVk}, the supports of the summands in \re{eqVDefn} are disjoint so \re{eqXVkContra} implies that
\[
X_V(1,S)=([0,24]\x\{0\})\cup\bigcup_{k=1}^\infty G_k\left(\frac{7}{8}S_{\alpha_k}\right).
\]

Using the notation of Proposition \ref{propOmegaTraj}, 
\[G_k(S_\alpha)\subset G_k(R_0)=(7/8)^{k-1}[Q+(24,0)],\] so combining that proposition with Corollary \ref{corLineTraj}, we conclude that
\[
X_{V}(2,S)=\frac{7}{8}S,
\] as required.
\end{proof}

As for the regularity of $V$, we have the following.
\begin{lemma}\label{factVunifCts}
$V$ is uniformly continuous in $\RR\x\RR^2$, and Lipschitz in $\RR\x H_+^\xi$ for any $\xi>0$. Moreover, $V$ is locally Lipschitz on $\RR\x(\RR^2\backslash\{(0,0)\})$.
\end{lemma}
\begin{proof}

Since $V(t,x)=\nu(t-1,x)$ for $t\in[1,2]$, and 
\[
0\equiv V_k(0,\cdot)\equiv V_k(1,\cdot)\equiv \nu(0,\cdot)\equiv \nu(1,\cdot)\equiv V(t,\cdot)
\] for all $k$, and $t\in\RR\backslash[0,2]$, it suffices to prove uniform continuity in $[0,1]\x\RR^2$. For $x=(x_1,x_2)$ with $x_1\geq (7/8)^{n-1}(24-\tfrac{3}{2})$, 
\[
V(t,x)=\sum_{k=1}^n \left(\frac{7}{8}\right)^{k-1}\mathcal{R}V_k(t,G_k^{-1}(x)),
\]
a finite sum of Lipschitz functions on $[0,1]\x\RR^2$. Whereas, for $x_1\leq (7/8)^{n-1}(24-\tfrac{3}{2})$, $|V(x,t)|\lesssim (n+2)(7/8)^{n}$.

Hence for any $\xi>0$, $V(t,x)$ is Lipschitz on $[0,1]\x H_+^\xi$ with constant depending on $\xi$, and $V(t,x)\to 0$ uniformly as $x_1\to 0$. Uniform continuity of $V$ follows.

Finally, by \re{eqSuppV1},  $\supp V\subset (\RR\x\bigcup_{n}H_+^{1/n}) \cup (\RR\x\{(0,0)\})$, so $V$ is locally Lipschitz on the required set.
\end{proof}

\subsubsection{Flowing \texorpdfstring{$S$}{S} to \texorpdfstring{$0$}{0}}
We can now prove Theorem \ref{thmFractal2}, from which Theorem \ref{thmPassiveScalar2} follows, as noted previously.
\begin{proof}[Proof (of Theorem \ref{thmFractal2})]
Using the velocity $V$ defined by \re{eqVDefn} we construct  a velocity field that flows $S$ to the origin in finite time. Recall that $\dim_{H}S=2$ by \re{eqLimDimSAlphaK}. We write
\[
\tilde V(t,x)\coloneqq\sum_{k=0}^\infty \left(\frac{63}{64}\right)^kV\left(\left(\frac{9}{8}\right)^k(t-t_k),\left(\frac{8}{7}\right)^kx\right),
\]
where $t_k\coloneqq18(1-(8/9)^k)=2\sum_{n=1}^k\left(\frac{8}{9}\right)^{n-1}$, for $k\in\NN$, and $t_0=0$. 

By Proposition \ref{thmXVscale}, 
\[
X_{\tilde V}(t_k,S)=\left(\frac{7}{8}\right)^kS
\]
and so $X_{\tilde V}(18,S)=\{0\}$.

Since $V$ is bounded and uniformly continuous, so is $\tilde V$ (which converges uniformly to $0$ as $t\to18$), by a similar argument to the proof of Lemma \ref{factVunifCts}. Furthermore, $\tilde V$ is Lipschitz on $\RR\x H_+^\xi$ for any $\xi>0$, and locally Lipschitz on $\RR\x(\RR^2\backslash\{(0,0)\})$ by Lemma \ref{factVunifCts}.

Since $V$ is weakly divergence free, so is $\tilde V$. 
\end{proof}

To complete this section, we prove the claim from the introduction regarding the Sobolev regularity of $V$.
\begin{proposition}\label{propVhatRegularity}
For all $p\in[1,\infty)$,
\[V\in L^\infty(0,2;W^{1,p}).\]
And $\tilde V\in L^q(0,18;W^{1,p})$ for $1\leq q\leq\infty$, $1\leq p<\infty$ with
\[
p(1-1/q)<\frac{\log(8/7)}{\log(9/8)}.
\]
\end{proposition}
\begin{proof}
For $p\in[1,\infty)$ and $k\in\NN$,
\[
\|\nabla V_k\|_{L^\infty(0,1;L^p)}\leq (k+1)\|\nabla U_{\alpha_k}\|_{L^\infty_{x,t}}\lesssim(k+1)\delta_k^{-1}.
\]
Hence, by \re{eqVDefn} and the properties of $\nu$, $V\in L^\infty(0,2;W^{1,p})$. This relies upon the fact that by the choice of $\alpha$ $\delta_k^{-1}=\mathcal{O}(k)$, and the fact that $p<\infty$ to obtain $W^{1,p}$ convergence of \re{eqVDefn}.

The claimed regularity for $\tilde V$ follows from arguments similar to those in Proposition \ref{propWregularity}.
\end{proof}

\section{Active-scalar systems: Slits in \texorpdfstring{$\RR^2$}{R2}}\label{secActiveScalars2D}
Motivated by considerations from fluid mechanics, in particular problems like the  vortex filament conjecture, we next investigate velocities $u$ that are not only divergence free, but also satisfy an active scalar system.

In this section we prove Theorem \ref{thmSlit} (stated in the next subsection), from which Theorem \ref{thmActiveScalar2D} follows, by Lemmas \ref{lemPassToWeak2D}  and \ref{lemReversibility}. We now recall the main result of this section:

\begingroup
\def\thetheorem{\ref{thmActiveScalar2D}}
\begin{theorem}
There exists a kernel $K\in C(\RR^2;\RR^2)$ and $T>0$ such that there exists a time-dependent measure $\omega$ that is weakly advected by 
\begin{equation}
\def\theequation{\ref{eqUdefn_abstractOmega2D}}
u(t,x)\coloneqq \int_{\RR^2} K(x-y)\,\d\omega(t)(y),
\end{equation}
and $\dim_H\supp\omega(0)=0$, $\dim_H\supp\omega(T)=1$.
\end{theorem}
\addtocounter{theorem}{-1}
\endgroup

\subsection{Lagrangian construction of the measure \texorpdfstring{$\omega$}{ω}}

Fixing a line segment $I=[-1,1]\x\{0\}$, we will consider an active scalar system in which the velocity is recovered by convoluting a continuous kernel $K\in C(\RR^2;\RR^2)$ with a pushforward of the 1-Hausdorff measure on $I$:
\begin{equation}\label{eqAS2weak}
u(t,x)=\int_I K(x-X(t,\alpha))\ \d \mathcal{H}^1_I(y).
\end{equation}
By virtue of Lemma \ref{lemPassToWeak2D}, if $X$ is also the trajectory map of $u$ given by \re{eqAS2weak} then $\omega(t)=X(t)\#\mathcal{H}^1_I$ is the weak solution of the measure-valued active scalar system \re{eqWeakAdvectionMeasure2D}. 
 In order that $u$ be weakly divergence free, it is enough that $K$ is given by a perpendicular gradient  \[K=\nabla^\perp\kappa=(-\p_2\kappa,\p_1\kappa),\] for some $\kappa\in C^1(\RR^2;\RR)$.
 
 \begin{definition}
 Given  $K\in C(\RR^2;\RR^2)$, we say $K$ \emph {admits a flow} (of $\omega_0$)
 \[X\in C([0,\infty)\x I;\RR^2),\quad\mathrm{differentiable\ with\ respect\ to\ } t\] if $u(t,x)$ given by \re{eqAS2weak} exists for all $x\in X(t,I)$ and $t\in[0,\infty)$, and $X$ is a trajectory map for $u$ on $I$, in the sense of Definition \ref{defTrajectoryMap}.
\end{definition}

The claimed construction can now be formalised in the following theorem.
\begin{theorem}\label{thmSlit}
There exists $\kappa\in C^{1,1/2}_c(\RR^2;\RR)$ such that $K=\nabla^\perp \kappa$ admits a flow $X$ with $X(0,x)=x$ for all $x\in I$, and $X(t,I)=\{0\}$ for all sufficiently large $t>0$.
\end{theorem}
\begin{proof}

Set
\[
\kappa(x)=x_1x_2|x|^{-1/2} \chi(|x|),
\]
where $\chi$ is a smooth cutoff function with $\chi(r)\equiv 1$ for $r\leq 2$ and $\chi(r)\equiv 0$ for $r\geq 3$. Note that $\kappa$ vanishes smoothly where it changes sign, except at $0$ where it is only $C^{1,1/2}$.

The corresponding kernel is 
\[
K(x)\coloneqq\nabla^\perp\kappa= \left(\begin{array}{c}-\frac{x_1}{|x|^{1/2}}+\frac{x_1x_2^2}{2|x|^{5/2}}\\
\frac{x_2}{|x|^{1/2}}-\frac{x_1^2x_2}{2|x|^{5/2}}
\end{array}\right)\chi(|x|)+x_1x_2|x|^{-3/2}\chi'(|x|)x^\perp.
\]
Observe that $K(x_1,0)=-(x_1|x_1|^{-1/2},0)$ for $x_1\in (-2,2)$. 

With the notation $J\coloneqq [-1,1]$, we now find a solution $Y\in C([0,\infty)\x J;\RR)$ of the integral equation
\begin{equation}\label{eqIntSys1}
Y(t,\alpha)=-\int_0^t \int_{J}\frac{Y(s,\alpha)-Y(s,\beta)}{|Y(s,\alpha)-Y(s,\beta)|^{1/2}} \ \d \beta\ \d s + \alpha,
\end{equation}
for all $\alpha\in J$.
We will later show that  $X(t,a)\coloneqq[Y(t,a_1),0]$ (for $a\in I$)  has the required properties.

We first check that a solution $Y$ to \re{eqIntSys1} exists. Let $M_\varepsilon\in C^\infty(\RR;\RR)$ be given by mollifying $K(x,0)$:
\begin{equation}\label{eqMEpsilonDefn}
M_\varepsilon(x_1)\coloneqq \rho_\varepsilon\ast K(\cdot,0)(x_1)
\end{equation}
where $\rho_\varepsilon(x_1)= \frac{1}{\varepsilon}\rho(x_1/\varepsilon)$ for an even mollifier $\rho\in C^\infty_c(\RR)$. Now consider the approximate system
\begin{equation}\label{eqIntSys2}
\frac{\d}{\d t}Z(t,\alpha)= \int_JM_\varepsilon(Z(t,\alpha)-Z(t,\beta)) \ \d \beta\ ,\qquad Z(0,\alpha)=\alpha.
\end{equation}
Since $\p_xM_\varepsilon,\ \p_x^2M_\varepsilon$ are bounded, the right-hand side of the ODE in \re{eqIntSys2} is Lipschitz in the norm ($\|\cdot\|_{W^{1,\infty}}$) for $Z(t, \cdot)\in C^1$. Hence there exists $T=T_\varepsilon>0$ and a unique solution $Y_\varepsilon\in C^1([0,T_\varepsilon]; C^1(J))$  to \re{eqIntSys2}. 

As $M_\varepsilon$ and $Y_\varepsilon(0,\cdot)$ are odd, $-Y_\varepsilon(t,-\alpha)$ is also a solution, so $Y_\varepsilon$ is odd by uniqueness. We deduce that $\int_JM_\varepsilon(-Y_\varepsilon(t,\beta)) \ \d \beta=0$. Now $M_\varepsilon$ is non-increasing, hence $\frac{\d}{\d t} Y_\varepsilon(t,\alpha)\leq 0$ for $\alpha$ such that $Y_\varepsilon(t,\alpha)\geq 0$ and vice-versa. Therefore $\|Y_\varepsilon(t) \|_{L^\infty}$ is bounded independent of $t$ and $\varepsilon$. 

We also have
\[
\p_t\p_\alpha Y_\varepsilon(t,\alpha)=\p_\alpha Y_\varepsilon(t,\alpha)\int_J \p_xM_\varepsilon(Y_\varepsilon(t,\alpha)-Y_\varepsilon(t,\beta))\ \d\beta,
\] 
which implies, given (negative) upper bounds on $\p_x M_\varepsilon$ in a suitable domain, that $\p_\alpha Y_\varepsilon(t,\alpha)\in(0,1]$ for all $(t,\alpha)\in[0,\infty)\x J$. We deduce that the solution may be continued indefinitely in $C^1$ for any $\varepsilon>0$. 

Given the above bounds on $\p_t Y_\varepsilon$ and $\p_\alpha Y_\varepsilon$, we see that $\{Y_\varepsilon\}_\varepsilon$ is equicontinuous in $(t,\alpha)$, hence we may assume, passing to a subsequence as necessary, that there exists $Y:[0,\infty)\x J\to\RR$ such that  $Y_\varepsilon\to Y$ uniformly as $\varepsilon\to 0$. Additionally, the limit is Lipschitz:
\[
Y\in C^{0,1}([0,\infty)\x J).
\]

Since $K(\cdot,0)$ is uniformly continuous, $M_\varepsilon$ converges uniformly to $K(\cdot,0)$, hence integrating \re{eqIntSys2} in time yields \re{eqIntSys1} for the limit $Y$. Directly from \re{eqIntSys1} we see that $Y$ is $C^1$ in time for every $\alpha\in J$.

Defining $X(t,\alpha)\coloneqq [Y(t,\alpha),0]$, uniform continuity and compact support of $K$ imply that $u$ defined by \re{eqAS2weak} is continuous. Combining \re{eqAS2weak} and \re{eqIntSys1} yields \re{eqDefnOfLagFlow}:
\[
u(X(t,\alpha))=\p_tX(t,\alpha).
\]

It remains to check that the image of $X$ collapses in finite time. For fixed $t\geq 0$, we have seen that $Y_\varepsilon$ is odd and non-decreasing with respect to $\alpha$ for every $\varepsilon>0$, hence so is $Y$. Similarly, $|Y|$ is non-increasing with respect to $t$ for any fixed $\alpha$. It therefore suffices to show that $Y(1,t)= 0$ for all sufficiently large times $t$. Now by monotonicity of $Y(t,\alpha)$ and $K(\alpha,0)$ (with respect to $\alpha$), \re{eqIntSys1} implies
\[
\p_tY(t,1)\leq -\sqrt{Y(t,1)}
\]
so (at least while $Y(t,1)>0$)
\[
Y(t,1)\leq (\sqrt{Y(0,1)}-t/2)^2,
\]
and in particular $Y(t,1)=0$ for $t\geq 2$.
\end{proof}

Now Theorem \ref{thmActiveScalar2D} follows by applying Lemma \ref{lemPassToWeak2D} to $\mu_0 = \mathcal{H}^1_I$ and Lemma \ref{lemReversibility}. Furthermore, since the support of the resulting measure is containted in a finite line segment (co-dimension 1) for all $t$ and $|\nabla K(x)|\lesssim |x|^{-1/2} $ the following proposition is straightforward.

\begin{proposition}\label{propSlitURegularity}
The velocity $u$ given by \re{eqAS2weak} for $K$ and $X$ as in Theorem \ref{thmSlit} satisfies
\[
u\in L^\infty(0,\infty;W^{1,p})
\]
for $p\in[1,2)$.
\end{proposition} 

\section{Active-scalar systems: Ribbons in \texorpdfstring{$\RR^3$}{R3}}\label{secActiveScalars3D}
In this section we adapt the 2D example in the previous section to prove Theorem \ref{thmTimeReversal}. That is, we present an active scalar system (with vortex-stretching) that admits a divergence-free weak measure-valued solution and where the support of the measure is initially a line segment but has Hausdorff dimension $2$ for sufficiently large positive times.

As before, we begin with a Lagrangian construction of a  system in which the image of a two-dimensional set collapses in finite time, then pass to a weak formulation and which is reversible in time.
\subsection{Lagrangian construction}
Let $\Xi\coloneqq \{0\}\x[-1,1]\x[-1,1]$ and fix the vector-valued measure $\omega_0$ given by
\[
\omega_0\coloneqq (0,1,0) \mathcal{H}^2_\Xi,
\]
(recall that  $\mathcal{H}^2_\Xi$ denotes the Hausdorff measure on $\Xi$). Recall also that in the vector-valued case we are concerned  with  the transport equations with stretching:
\begin{equation}\label{eqAS13D}
\p_t \omega +(u\cdot\nabla)\omega=(\omega\cdot\nabla)u.
\end{equation}
In Lagrangian coordinates this system is equivalent to
\begin{equation}\label{eqAS3D_Lag}
\omega(X_u(t,\alpha))=\nabla_\alpha X_u(t,\alpha)\omega_0(\alpha).
\end{equation}
 
Hence we will consider a velocity $u$ recovered from a flow $X$ via
\begin{equation}\label{eqAS2Weak3D}
u(t,x)=\int_\Xi K(x-X(t,\alpha))\p_{\alpha_2} X(t,\alpha)\ \d \alpha,
\end{equation}
for some continuous matrix-valued kernel $K$.
\begin{definition}
A kernel $K\in C(\RR^3;\RR^{3\x 3})$ \emph{admits a flow} $X\in C([0,\infty)\x\Xi;\RR^3)$ of $\omega_0$ if $X$ is differentiable with respect to $t$ and $\alpha_2$,  the corresponding $u$, defined by \re{eqAS2Weak3D} exists on $\{(t,x)\,\colon t\in[0,\infty),\, x\in X(t,\Xi)\}$, and 
\begin{equation}\label{eqASLag3D}
\p_t X(t,a)=u(t,X(t,a))\quad X(0,a)=a
\end{equation}
for all $(t,a)\in[0,\infty)\x\Xi$, i.e.\ $X=X_u$ is a trajectory map for $u$.
\end{definition}

The main construction of this section is contained in the following Theorem.
\begin{theorem}\label{thmRibbon}
There exists a kernel $K\in C(\RR^3;\RR^{3\x3})$ admitting a flow $X$ of $\omega_0$ such that $X(t,\Xi)\subset\{0\}\x[-1,1]\x\{0\}$ for sufficiently large times $t>0$.
\end{theorem}
\begin{proof}
Following the previous section, let $\kappa(x)=x_1x_3|(x_1,x_3)|^{-1/2}\chi(|x|)$ for a smooth cutoff function $\chi$, such that $\chi(r)=1$ for $r\in[0,3)$. Now set
\[
K(x)=\begin{pmatrix} 0 & \p_3\kappa & 0\\
0&0&0\\
0&-\p_1\kappa&0\end{pmatrix}.
\]

Let $Y:[0,\infty)\x[-1,1]\to\RR$ be a solution of 
\begin{equation}\label{eqY_DE2}
\p_tY(t,\alpha)=-2\int_{-1}^1\frac{Y(t,\alpha)-Y(t,\beta)}{|Y(t,\alpha)-Y(t,\beta)|^{1/2}}\,\d\beta\quad Y(0,\alpha)=\alpha,
\end{equation}
constructed following the steps of the previous section. Fix $X\in C([0,\infty)\x\Xi;\RR^3)$ defined by
\[
X(t,\alpha)\coloneqq (0,\alpha_2,Y(t,\alpha_3))\in\Xi.
\]
Now $\p_{\alpha_2}X=(0,1,0)$ is defined for almost all $\alpha\in\Xi$ and so, for $x\in\Xi$, \re{eqAS2Weak3D} becomes:
\[
u(t,x)=\int_\Xi \begin{pmatrix}\p_3\kappa (x-X(\alpha))\\0\\-\p_1\kappa(x-X(\alpha))\end{pmatrix}\,\d \alpha.
\]
By the choice of $\chi$, for $x,y\in \Xi$ 
\[
\kappa(x-y)=\frac{(x_1-y_1)(x_3-y_3)}{|(x_1-y_1,x_3-y_3)|^{1/2}},
\]
and since  $x_1=X_1(\alpha)=0$ for $\alpha\in\Xi$, $\p_3\kappa\big|_{\{0\}\x\RR^2}\equiv 0$. Hence to verify \re{eqASLag3D}, it remains to check that 
\begin{multline*}
\p_tX(t,\alpha)=\int_\Xi \begin{pmatrix}0\\0\\-\p_1\kappa(X(t,\alpha)-X(t,\beta))\end{pmatrix}\,\d \beta\\
=\int_\Xi \begin{pmatrix}0\\0\\-\frac{X_3(t,\alpha)-X_3(t,\beta)}{|X_3(t,\alpha)-X_3(t,\beta)|^{1/2}}\end{pmatrix}\,\d \beta\\
=-2\int_{-1}^1 \begin{pmatrix}0\\0\\\frac{Y(t,\alpha_3)-Y(t,\beta_3)}{|Y(t,\alpha_3)-Y(t,\beta_3)|^{1/2}}\end{pmatrix}\,\d \beta_3.
\end{multline*}
This is indeed the case, by \re{eqY_DE2}.
\end{proof}

As in the previous section we obtain the following regularity for $u$.
\begin{proposition}\label{propRibbonURegularity}
The velocity $u$ given by \re{eqAS2Weak3D} for $K$ and $X$ as in Theorem \ref{thmRibbon} satisfies
\[
u\in L^\infty(0,\infty;W^{1,p}(\RR^3))
\]
for $p\in[1,2)$.
\end{proposition} 

\subsection{Weak formulation and time reversal}
In order to reverse this solution in time, we will pass to the weak form of \re{eqAS13D}, in the sense of Definition \ref{defWeakAdvectionMeasure}, i.e.\ 
\begingroup
\def\theequation{\ref{eqWeakAdvectionMeasure}}
\begin{multline}
\int_0^T\int\p_t\phi(t,x)\cdot\,\d\omega(t)(x)\d t- \int \phi(T,x)\cdot\,\d\omega(T)(x) \\
+  \int \phi(0,x)\cdot\,\d\omega(0)(x)+\int_0^T\int((u\cdot\nabla)\phi)\big|_{(t,x)}\cdot\,\d\omega(t)(x)\d t\\
-\int_0^T\int ((\nabla\phi)^\top u)\big|_{(t,x)}\cdot\,\d\omega(t)(x)\d t - \int_0^T\langle\phi\cdot u,\nabla\cdot\omega(t)\rangle\d t =0,
\end{multline}
for all $\phi\in C^\infty_c([0,T]\x\RR^3;\RR^3)$.
\addtocounter{equation}{-1}
\endgroup

Using $X$ from Theorem \ref{thmRibbon}, we first define a measure that is weakly advected by $u$, whose support has the required behaviour. Indeed, we define a vector-valued measure $\omega(t)$ by
\begin{equation}\label{eqOmegaFromX}
\omega(t)(A)=\int_{X^{-1}(t,A)}\p_{\alpha_2}X(t,x)\,\d\mathcal{H}^2_{\Xi}(x),
\end{equation}
for any Borel set $A$. This is well-defined since, for $t\in[0,\infty)$, $X^{-1}(t,A)\subset\Xi$, where $\p_{\alpha_2}X$ exists (on a set of full $\mathcal H^2$ measure). In particular, for a Borel function $f:\RR^3\to\RR^3$,
\begin{equation}\label{eqFDotOmega}
\int f(x)\cdot\,\d\omega(t)(x)= \int_\Xi f(X(t,x))^\top\p_{\alpha_2}X(t,x)\,\d\mathcal{H}^2_\Xi(x).
\end{equation}

\begin{lemma}\label{lemPassToWeak}
Let $u$ and $X$ be as in the proof of Theorem \ref{thmRibbon}, and $\omega$  as in \re{eqOmegaFromX} then the distributional divergence of $\omega$ is a finite Borel measure and $\omega$ is weakly advected by $u$ on any time interval $[0,T]$, $T>0$ in the sense of Definition \ref{defWeakAdvectionMeasure}.
\end{lemma}
\begin{proof}
Fix $T>0$. We first calculate the distributional divergence of $\omega(t)$ for any $t\in[0,T]$. Let $\varphi\in C^\infty_c(\RR^3;\RR)$ then by definition $\nabla\cdot\omega(t)$ satisfies
\[
\langle\varphi,\nabla\cdot\omega(t)\rangle = -\int \nabla\varphi(x)\cdot\,\d\omega(t)(x).
\]
Now 
\begin{multline*}
\int \nabla\varphi(x)\cdot\,\d\omega(t)(x)=\int \nabla\varphi\big|_X(t,\alpha)\cdot\p_{\alpha_2}X(t,\alpha)\,\d\omega_0(\alpha)\\
=\int\p_{\alpha_2}[\varphi(X(t,\alpha))]\cdot\,\d\omega_0(\alpha),
\end{multline*}
but since $\omega_0$ is $(0,1,0)\mathcal{H}^2_\Xi$, we have

\begin{equation}\label{eqPhiInnerDivOmega}
\int \nabla\varphi(x)\cdot\,\d\omega(x)=\int_{-1}^1\varphi(X(t,(0,1,\alpha_3)))-\varphi(X(t,(0,-1,\alpha_3)))\,\d\alpha_3.
\end{equation}
Hence the distributional divergence of $\omega(t)$ is indeed a finite Borel (signed) measure.

Let $X_\varepsilon:[\varepsilon,T-\varepsilon]\x\Xi_\varepsilon\to\RR^3$ denote the mollified trajectory map, restricted to $[\varepsilon,T-\varepsilon]\x\Xi_\varepsilon$, where $\Xi_\varepsilon=\{0\}\x[\varepsilon-1,1-\varepsilon]\x [\varepsilon-1,1-\varepsilon]$:
\[
X_\varepsilon(t,\alpha)=\int_0^T\int_{\Xi} \rho_\varepsilon(|t-s, |\alpha-\beta||) X(s,\beta)\,\d\beta\d s
\]
for a standard family of mollifiers $\rho_\varepsilon:\RR\to\RR$.

We now show that $X_\varepsilon$ satisfies an approximate version of \re{eqWeakAdvectionMeasure}. Indeed for $\phi\in C^\infty_c([0,T]\x\RR^3;\RR^3)$
\begin{multline}\label{eqA1A2}
\int_\varepsilon^{T-\varepsilon}\int_{\Xi_{\varepsilon}}\p_t\phi^\top\big|_{X_\varepsilon}\p_{\alpha_2}X_\varepsilon \,\d\mathcal{H}^2_\Xi\d t= \int_\varepsilon^{T-\varepsilon}\int_{\Xi_{\varepsilon}}\p_t[\phi^\top(X_\varepsilon)]\p_{\alpha_2}X_\varepsilon\,\d\mathcal{H}^2_\Xi\d t \\
-  \int_\varepsilon^{T-\varepsilon}\int_{\Xi_{\varepsilon}}(\p_tX_\varepsilon\cdot\nabla)\phi^\top\p_{\alpha_2}X_{\varepsilon}\,\d\mathcal{H}^2_\Xi\d t
\eqqcolon A_1-A_2,
\end{multline}
and
\begin{multline*}
A_1=\int_\varepsilon^{T-\varepsilon}\int_{\Xi_{\varepsilon}}\p_t[\phi^\top(X_\varepsilon)]\p_{\alpha_2}X_\varepsilon\,\d\mathcal{H}^2_\Xi\d t\\
= \int_{\Xi_{\varepsilon}}\phi^\top(T-\varepsilon,X_\varepsilon(T-\varepsilon,\alpha))\p_{\alpha_2}X_\varepsilon(T-\varepsilon,\alpha)\,\d\mathcal{H}^2_\Xi(\alpha) \\
- \int_{\Xi_{\varepsilon}}\phi^\top(\varepsilon,X_\varepsilon(\varepsilon,\alpha))\p_{\alpha_2}X_\varepsilon(\varepsilon,\alpha)\,\d\mathcal{H}^2_\Xi(\alpha) \\
-\int_\varepsilon^{T-\varepsilon}\int_{\Xi_{\varepsilon}}\phi^\top(X_\varepsilon)\p_t\p_{\alpha_2}X_\varepsilon\,\d\mathcal{H}^2_\Xi\d t\eqqcolon B_1-B_2-A_3.
\end{multline*}
Integrating the final term by parts again, now with respect to $\alpha_2$ yields
\begin{multline*}
A_3=\int_\varepsilon^{T-\varepsilon}\int_{\Xi_{\varepsilon}}\phi^\top(X_\varepsilon)\p_t\p_{\alpha_2}X_\varepsilon\,\d\mathcal{H}^2_\Xi\d t\\
=\int_\varepsilon^{T-\varepsilon}\int_{\varepsilon-1}^{1-\varepsilon}\phi^\top(X_\varepsilon(t,(0,1-\varepsilon,\alpha_3)))\p_tX_\varepsilon(t,(0,1-\varepsilon,\alpha_3))\\
-\phi^\top(X_\varepsilon(t,(0,\varepsilon-1,\alpha_3)))\p_tX_\varepsilon(t,(0,\varepsilon-1,\alpha_3))  \,\d\alpha_3\d t\\
-\int_\varepsilon^{T-\varepsilon}\int_{\Xi_{\varepsilon}}(\p_{\alpha_2}X_\varepsilon\cdot\nabla)\phi^\top\big|_{X_\varepsilon}\p_tX_\varepsilon\,\d\mathcal{H}^2_\Xi\d t.
\end{multline*}
Now  $\p_tX_\varepsilon$, $\p_{\alpha_2}X_\varepsilon$, and $X_\varepsilon$ are uniformly bounded and converge pointwise to $u(X)$, $\p_{\alpha_2}X$ and $X$, respectively. Therefore by \re{eqFDotOmega}, the left-hand side of \re{eqA1A2} converges as $\varepsilon\to 0$:
\[
\int_\varepsilon^{T-\varepsilon}\int_{\Xi_{\varepsilon}}\p_t\phi^\top\big|_{X_\varepsilon}\p_{\alpha_2}X_\varepsilon \,\d\mathcal{H}^2_\Xi\d t\to\int_0^T\int\p_t\phi(t,x)\cdot\,\d\omega(t)(x)\d t.
\]
Furthermore,
\[
B_1-B_2 \to \int\phi(T,x)\cdot\,\d\omega(T)(x) - \int\phi(0,x)\cdot\,\d\omega(0)(x),
\]
\[
A_3\to -\int_0^T\langle(\phi\cdot u)\big|_{(t,\cdot)}, \nabla\cdot\omega(t)\rangle\,\d t -\int_0^T\int ((\nabla\phi)^\top u)|_{(t,x)}\cdot\,\d\omega(t)(x)\,\d t,
\] where we have used \re{eqPhiInnerDivOmega}, and
\[
A_2\to\int_0^T\int((u\cdot\nabla)\phi)|_{(t,x)}\cdot\,\d\omega(t)(x)\d t.
\]

Thus $\omega$ is indeed weakly advected by $u$, in the sense of \re{eqWeakAdvectionMeasure}.
\end{proof}

We can now prove the main result of this section (which we first recall) by applying a time reversal argument to the solutions given in Theorem \ref{thmRibbon} in weak form.
\begingroup
\def\thetheorem{\ref{thmTimeReversal}}
\begin{theorem}
There exists a kernel $K\in C(\RR^3;\RR^{3\x3})$, $T>0$, and a time-dependent and locally finite measure $\tilde\omega$ with locally finite distributional divergence that is weakly advected by 
\begin{equation}
\def\theequation{\ref{eqUdefn_abstractOmega}}
\tilde u(t,x)\coloneqq \int K(x-y)\,\d\tilde\omega(t)(y),
\end{equation}
such that $\dim_H\supp\tilde\omega(0)=1$, $\dim_H\supp\tilde\omega(T)=2$, and $\tilde u$ is weakly divergence free.
\end{theorem}
\addtocounter{theorem}{-1}
\endgroup
\begin{proof}
Let $\omega$ be given by \re{eqOmegaFromX}, associated to the trajectory map given by Theorem \ref{thmRibbon}, and let $u$ be the associated velocity from the proof. Let $T$ be sufficiently large that 
\[\dim_H\supp\omega(T)=1.\]
Let $K$ be the kernel given by Theorem \ref{thmRibbon}, and define 
\[\tilde\omega(t)\coloneqq -\omega(T-t).\]
Now it is straightforward to check that $\tilde\omega$ is weakly advected by $\tilde u(t,x)\coloneqq- u(T-t)$, that \re{eqUdefn_abstractOmega} holds, and that $\supp\tilde\omega(t)$ has the required evolution.
\end{proof}

\appendix
\section{Appendix}
\begin{lemma}\label{factXrescale}
Suppose that  $U\in C([0,T]\x\RR^2;\RR^2)$ admits a unique trajectory map $X_U(t,x)$, then for $G\in GL_2(\RR)$ and $r\in\RR^2$,
\begin{equation}\label{eqTrajLin}
X_{GU\circ G_r^{-1}}(t,x)=G_rX_U(t,G_r^{-1}x),
\end{equation}
where $G_r$ denotes the affine map $x\mapsto G x + r$, and $GU\circ G_r^{-1}(t,x)$ denotes $GU(t,G_r^{-1}x)$. This trajectory map is uniquely defined.

\end{lemma}
\begin{proof}
Note that for $x\in\RR^2$
\begin{multline*}
\frac{\d}{\d t}G_rX_U(t,G_r^{-1}x)=GU(t,X_U(t,G_r^{-1}x))\\=(GU\circ G_r^{-1})(t,G_rX_U(t,G^{-1}_rx)),
\end{multline*}
and for $x\in\RR^2$
\[
G_rX_U(0,G^{-1}_rx)=x,
\]
i.e.\ $G_rX_U(t,G_r^{-1}x)$ is the trajectory map for $(GU\circ G_r^{-1})$ (uniqueness follows from the uniqueness hypothesis for $X_U$, and invertibility of $G_r$).
\end{proof}

\section*{Acknowledgements}
We would like to thank Simon Baker for several helpful discussions.
BCP and JR are partially supported by European Research Council, ERC Consolidator Grant no. 616797.


\end{document}